\documentclass[journal,twoside,web]{ieeecolor}

\usepackage{generic}
\usepackage{textcomp}

\usepackage{array}
\usepackage{ragged2e} 
\usepackage{makecell} 

\def\BibTeX{{\rm B\kern-.05em{\sc i\kern-.025em b}\kern-.08em
    T\kern-.1667em\lower.7ex\hbox{E}\kern-.125emX}}
\usepackage{amsmath}
\usepackage{multirow} 
\usepackage[utf8]{inputenc} 
\usepackage[T1]{fontenc}    
\usepackage{hyperref}       
\usepackage{url}            
\usepackage{booktabs}       
\usepackage{amsfonts}       
\usepackage{nicefrac}       
\usepackage{microtype}      
\usepackage{amsfonts}
\usepackage{amssymb}
\usepackage{graphicx}
\usepackage{fancyhdr,color}
\usepackage{mathtools}
\usepackage{biblatex}
\usepackage{bbm}
\usepackage{bm}
\usepackage{mathabx}
\usepackage{algorithm}
\usepackage{algpseudocode}
\usepackage{enumerate} 
\usepackage{xcolor}
\usepackage{balance}

\newcommand*{\medcap}{\mathbin{\scalebox{0.8}{\ensuremath{\bigcap}}}}%
\newtheorem{theorem}{Theorem}
\newtheorem{remark}{Remark}

\newtheorem{definition}{Definition}
\newtheorem{assumption}{Assumption}
\newtheorem{proposition}{Proposition}
\newtheorem{corollary}{Corollary}

\graphicspath{ {figures/} }

\begin{document}
\title{Approximate solution of stochastic infinite horizon optimal control problems for constrained linear uncertain systems}
\author{Eunhyek Joa, \IEEEmembership{Member, IEEE} and Francesco Borrelli, \IEEEmembership{Fellow, IEEE}
\thanks{The authors are with the Model Predictive Control Lab,
Department of Mechanical Engineering, University of California at Berkeley (e-mail: e.joa@berkeley.edu; fborrelli@berkeley.edu). }}

\maketitle

\noindent


\begin{abstract}
We propose a Model Predictive Control (MPC) with a single-step prediction horizon to approximate the solution of infinite horizon optimal control problems with the expected sum of convex stage costs for constrained linear uncertain systems.
The proposed method aims to enhance a given sub-optimal controller, leveraging data to achieve a nearly optimal solution for the infinite horizon problem.
The method is built on two techniques.
First, we estimate the expected values of the convex costs using a computationally tractable approximation, achieved by sampling across the space of disturbances.
Second, we implement a data-driven approach to approximate the optimal value function and its corresponding domain, through systematic exploration of the system's state space.
These estimates are subsequently used to calculate the terminal cost and terminal set within the proposed MPC.
We prove recursive feasibility, robust constraint satisfaction, and convergence in probability to the target set. 
Furthermore, we prove that the estimated value function converges to the optimal value function in a local region.
The effectiveness of the proposed MPC is illustrated with detailed numerical simulations and comparisons with a value iteration method and a Learning MPC that minimizes a certainty equivalent cost.
\end{abstract}

\begin{IEEEkeywords}
Data-driven, Learning, Optimal control, Predictive control.
\end{IEEEkeywords}

\section{Introduction}
\label{sec:introduction}

This paper proposes a computationally tractable approach to compute a solution to the following problem
\begin{equation} \label{eq: general format of problem}
\begin{aligned}
    &  Q^{\star}(x_S) = \min_{\pi(\cdot)} ~\,\mathbb{E}_{w_{0:\infty}}\Biggr[\sum_{k=0}^\infty \ell(x_k, ~ \pi(x_k)) \Biggr] \\
    & \qquad \qquad ~~ \textnormal{s.t.,} ~~ x_{k+1}  = A x_k + B \pi(x_k) + w_k, \\
    & \qquad \qquad  \qquad ~~\, x_0 = x_S, ~ x_k \in \mathcal{X}, ~ \pi(x_k) \in \mathcal{U},
\end{aligned}
\end{equation}
in a region $x_S$ of the state space.
In~(\ref{eq: general format of problem}), $w_k$ is an independently and identically distributed (i.i.d.) random variable with bounded support and a known probability distribution function. We assume the existence of a zero-cost terminal set $\mathcal{O}$; once the system reaches $\mathcal{O}$, it remains in there at no further cost. 
This infinite horizon optimal control problem (OCP) is similar to the \textit{Stochastic Shortest Path} problem in \cite[Chap. 3]{bertsekas2015dynamic} in having a cost-free terminal set, but differs because we deal with discrete-time linear uncertain systems with state and input constraints.

A classical way to solve this problem \eqref{eq: general format of problem} is Dynamic programming (DP) through value iteration \cite[Chap. 5.3.1]{bertsekas2012dynamic}, \cite[Chap. 3.4]{bertsekas2015dynamic},\cite[Chap. 4.4]{sutton2018reinforcement} or policy iteration \cite[Chap. 5.3.2]{bertsekas2012dynamic}, \cite[Chap. 3.5]{bertsekas2015dynamic}, \cite[Chap. 4.3]{sutton2018reinforcement}.
These methods involve gridding the state spaces to obtain the optimal value function and the associated DP policy globally, i.e., covering all possible states. 

Gridding suffers from the \textit{curse of dimensionality} \cite{bertsekas2012dynamic}. Approximate DP addresses this issue by approximating the value function and the policy \cite[Chap. 6]{bertsekas2012dynamic}.
For instance, Neural Networks (NN), which automatically extracts features from data, provide an alternative for approximating the value function \cite[Chap. 6.4.2]{bertsekas2012dynamic} and the policy \cite[Chap. 6.7]{bertsekas2012dynamic}.
However, when solving the constrained OCPs \eqref{eq: general format of problem} of safety-critical systems, training NNs are difficult \cite{li2018safe}. The resulting NNs require additional safety measures like safety filter \cite{wabersich2018linear}, which can result in conservative closed-loop performances.


Model Predictive Control (MPC) can be seen as a variation of approximate DP, where the terminal cost is an approximated value function \cite{bertsekas2022newton}. Researchers in \cite{cannon2012stochastic, farina2015approach, hewing2020recursively, lorenzen2016constraint} solve problem \eqref{eq: general format of problem} using MPC, assuming the stage cost is quadratic and that the associated probability distribution has known finite first and second moments. This allows them to reformulate the expected stage cost as a certainty equivalent cost and calculate the value function by solving the Riccati equation in a local region. However, without these assumptions, explicitly calculating the value function becomes challenging.
For deterministic systems, a data-driven approach to calculating a value function is presented in \cite{rosolia2017learning_journal}. 
In \cite{rosolia2017learning_journal}, from historical data, the value functions are computed and are utilized as the terminal cost and terminal set of the MPC, called Learning MPC (LMPC).
\cite{rosolia2017learning_journal} proves that this LMPC is an optimal solution for given infinite horizon OCPs for deterministic systems.

In this paper, we propose a novel data-driven MPC approach that provides an approximate solution to \eqref{eq: general format of problem} within a local region in the state space explored by the proposed algorithm.
In contrast to \cite{cannon2012stochastic, farina2015approach, hewing2020recursively, lorenzen2016constraint}, we consider a convex cost that is not necessarily reformulated as a certainty equivalent cost. Compared to \cite{rosolia2017learning_journal, rosolia2018stochastic, rosolia2021robust}, we consider the expected sum of stage costs instead of a certainty equivalent cost.


\begin{table*}[ht]
    \centering
    \scriptsize 
    \caption{Theoretical Comparison of Adaptive Control and Safe Reinforcement Learning (RL)}
    \renewcommand{\arraystretch}{1.3} 
    \setlength{\tabcolsep}{4pt} 
    \begin{tabular}{|>{\raggedright\arraybackslash}m{1.4cm}||%
                    >{\raggedright\arraybackslash}m{2.2cm}|%
                    >{\raggedright\arraybackslash}m{2.4cm}|%
                    >{\raggedright\arraybackslash}m{2.5cm}|%
                    >{\raggedright\arraybackslash}m{1.8cm}|%
                    >{\raggedright\arraybackslash}m{2.5cm}|%
                    >{\raggedright\arraybackslash}m{3.1cm}|}
        \hline
        \textbf{Methods} & \textbf{System Type} & \textbf{Constraints} & \textbf{Disturbance} & \textbf{Cost Function} & \textbf{Optimality} & \textbf{Stability} \\ \hline
        \vspace{0.2cm} \makecell[l]{BLF \cite{min2020adaptive,wu2019robust}} 
            & \makecell[l]{Nonlinear Uncertain,\\Unknown Parameters} 
            & \makecell[l]{Box on State \\ \cite{min2020adaptive}: Time-invariant \\ \cite{wu2019robust}: Time-varying \\ \& Input Saturation} 
            & \makecell[l]{\cite{min2020adaptive}: Wiener Process \\ \cite{wu2019robust}: Bounded with \\ Unknown Bound} 
            & \makecell[l]{\\ Not considered}
            & \makecell[l]{\\ Not considered}
            & \makecell[l]{\\ \cite{min2020adaptive}: Finite-time in Probability \\ \cite{wu2019robust}: Uniform Boundedness} \\ \hline
        Safe RL \cite{zanon2020safe} 
            & \makecell[l]{Affine,\\Unknown Parameters} 
            & \makecell[l]{Linear on \\ State \& Input} 
            & Not considered 
            & Quadratic 
            & \makecell[l]{No Guarantee, \\ RL for optimality.}
            & Asymptotic \\ \hline
        L-COCP \cite{agrawal2020learning} 
            & \makecell[l]{Affine,\\Unknown Parameters} 
            & \makecell[l]{Convex on \\ State \& Input} 
            & Not considered. 
            & Convex 
            & No Guarantee, Gradient descent for optimality. 
            & No Guarantee, as recursive feasibility is not ensured. \\ \hline
        Ours 
            & \makecell[l]{Linear Uncertain,\\Known Parameters} 
            & \makecell[l]{Linear on \\ State \& Input} 
            & \makecell[l]{Stochastic, Bounded,\\State Additive} 
            & \makecell[l]{Convex with \\ Assumptions} 
            & Near-Optimal Guarantee
            & \makecell[l]{Asymptotic in Probability} \\ \hline
    \end{tabular}
    \label{tab:comparison}
\end{table*}
In the context of data-driven control policy adaptation, our method is theoretically compared with adaptive control and safe reinforcement learning (RL), as summarized in Table \ref{tab:comparison}. Note that this review is not exhaustive due to page limit. The key distinction in leveraging data in this work is its focus on ensuring near-optimality rather than addressing unknown dynamics. Barrier Lyapunov Function (BLF)-based adaptive control \cite{min2020adaptive,wu2019robust} adjusts BLF parameters to maintain stability but only guarantees feasibility. Safe RL \cite{zanon2020safe} employs RMPC to learn both policy and model parameters but does not guarantee optimality. Learning Convex Optimization Control Policies (L-COCP) \cite{agrawal2020learning} tunes parameters using data but lacks stability guarantees for state-constrained problems due to the absence of recursive feasibility. In contrast, our method ensures both near-optimal performance and recursive feasibility, addressing limitations in existing approaches.

Our data-driven approach uses a forward iterative strategy, unlike DP, which works backward from a target set $\mathcal{O}$ to the initial state. 
Our method moves forward from the initial state to $\mathcal{O}$ and iteratively updates the value function and its domain. 
In particular, we collect data from each \textit{episode}, a task starting from an initial state and ending when it reaches a predefined target set $\mathcal{O}$. 
After each \textit{episode}, we explore areas in the state space that the system has not visited before, a process denoted as \textit{Exploration}. Using data from completed episodes and the \textit{Exploration}, we update the estimated value function and its domain. These updated values serve as the terminal cost and terminal set in the MPC for the next \textit{episode}.
Our contributions are summarized as:
\begin{itemize}
    \item  We formulate a one-step MPC which approximates \eqref{eq: general format of problem}. 
    \item We prove that the closed-loop system controlled by the proposed one-step MPC has the following properties: \textit{(i) convergence in probability:} the probability that the closed-loop system converges to the target set $\mathcal{O}$ approaches one as the time step goes to infinity, and \textit{(ii)} robustly satisfies state and input constraints.
    \item We prove that an estimated value function monotonically decreases as the number of episodes increases. Moreover, we prove that this estimated value function at $x_S$ converges to a neighborhood of the optimal value function of \eqref{eq: general format of problem} at $x_S$, where the radius of this neighborhood is proportional to the error in approximating the expectation.
    \item We propose a novel exploration method. Compared to the \cite[Algorithm 2]{bujarbaruah2022robust} where the farthest reachable states along certain predetermined or random directions are found, we calculate the direction of improving the performance and find the farthest reachable states along that direction.
    \item With numerical simulations, we compare our proposed approach with the LMPC \cite{rosolia2018stochastic} and the value iteration method in \cite{yang2020convex}. In terms of the expected total cost, our proposed approach outperforms the LMPC \cite{rosolia2018stochastic} by $13.75 \%$. Additionally, the estimated value function of our proposed method closely approximates the value function of the value iteration \cite{yang2020convex} in a local region, while the proposed method is 19 times faster than \cite{yang2020convex}.
\end{itemize}

\textit{Notation:} 
Throughout the paper, we use the following notation. 
The Minkowski sum of two sets is denoted as \(\mathcal{X} \oplus \mathcal{Y} = \{x+y: x \in \mathcal{X}, y \in \mathcal{Y}\}\).
The Pontryagin difference between two sets is defined as \(\mathcal{X} \ominus \mathcal{Y} = \{x\in\mathcal{X}: x+y \in \mathcal{X}, \forall y \in \mathcal{Y}\}\).
The m-th column vector of a matrix \(H\) is denoted as \([H]_m\). The m-th component of a vector \(h\) is \([h]_m\). 
\(\mathbb{P}(\mathcal{A})\) is the probability of the event \(\mathcal{A}\), and \(\mathbb{E}[\cdot]\) is the expectation of its argument. 
The notation \(x_{l:m}\) means the sequence of the variable \(x\) from time step \(l\) to time step \(m\).
$\mathrm{dim}(v)$ denotes the dimension of the vector $v$. $\mathrm{int}(\mathcal{S})$ denotes the interior of the set $\mathcal{S}$. $\mathrm{conv}\{\mathcal{S}\}$ denotes the convex hull of the set $\mathcal{S}$. 

\section{Problem setup}
\subsection{System Dynamics and Constraints}
We consider an uncertain linear time-invariant (LTI) system perturbed by a state additive, stochastic disturbance as follows:
\begin{equation} \label{eq:system}
\begin{split}
    & x_{k+1}^j  = A x_k^j  + B u_k^j  + w_k^j , ~ w_k^j \sim p(w), \\
    & x_0^j = x_S,
\end{split}
\end{equation}
where \(x_k^j  \in \mathbb{R}^{n_x}\) is the state, and \(u_k^j  \in \mathbb{R}^{n_u}\) is the input at time step $k$ of the episode $j$. System matrices \(A\) and \(B\) are known. 
At each time step \(k\) of the episode $j$, the system is affected by a random disturbance \(w_k^j  \in \mathbb{R}^{n_x}\) with known distribution $p(w)$. 
We use the following assumptions on the system \eqref{eq:system}.
\begin{assumption} \label{assum: bounded noise}
(Bounded Random Disturbance) We assume that the random disturbance $w_k^j$ is an i.i.d., zero-mean random variable with known distribution $p(w)$ and known support $\mathcal{W}$. The set $\mathcal{W}=\{w~|~H_w w \leq h_w\}$ is a nonempty polytope that contains the origin in its interior. 
Furthermore, $l_w$ is the number of vertices of $\mathcal{W}$, and $v_{i}^\mathcal{W}$ is the $i$-th vertex of the $\mathcal{W}$.
\end{assumption} 

System \eqref{eq:system} is subject to the following constraints:
\begin{equation} \label{eq:constraints}
\begin{split}
    & x_{k}^j \in \mathcal{X}, ~u_k^j \in \mathcal{U}, ~\forall w_k^j \in \mathcal{W},\\
    & \forall k \geq 0,~\forall j \geq 0,
\end{split}
\end{equation}
where \(\mathcal{X}=\{x\,|\, H_x x \leq h_x\}\) is a polyhedron, and \(\mathcal{U} = \{u \,|\,H_u u \leq h_u\}\) is a nonempty polytope.
For the feasibility of the problem setup, we use the following assumption.
\begin{assumption} \label{assum: feasibility of problem setup}
     $x_0^j = x_S \in \mathcal{X}$.
\end{assumption}

\subsection{Optimal Control Problem (OCP)}
We consider a control task where we would like to steer the system \eqref{eq:system} towards the target set $\mathcal{O} \subset \mathcal{X}$ while minimizing the expected sum of the stage cost $\ell(\cdot, \cdot): \mathcal{X} \times \mathcal{U} \rightarrow \mathbb{R}$ and robustly satisfying the constraints \eqref{eq:constraints} for every realization of additive disturbances.
We use the following assumptions on the target set $\mathcal{O}$ and the stage cost $\ell(\cdot, \cdot)$. 
\begin{assumption} \label{assum: Robust Positive Invariance}
(Robust Positive Invariant Set $\mathcal{O}$) A set $\mathcal{O} \subseteq \mathcal{X}$ is a robust positive invariant set for the autonomous system $x_{k+1}=(A+BK) x_{k} + w_k$ where $w_k \in \mathcal{W}$ and $K$ is a given state feedback gain such that $A+BK$ is Hurwitz, i.e., if $x_k \in \mathcal{O} \Rightarrow x_{k+1} \in \mathcal{O}, ~ \forall w_k \in \mathcal{W}$.
Furthermore, the set $\mathcal{O}$ is a polytope with a nonempty relative interior \cite{bertsekas2009convex} where its vertices are $\{v_1^\mathcal{O}, v_2^\mathcal{O}, ... , v_{l_\mathcal{O}}^\mathcal{O}\}$.
\end{assumption}

We define $\mathcal{KO} = \{u | u = Kx, x \in \mathcal{O}\}$.
By Assumption \ref{assum: Robust Positive Invariance}, $\forall x \in \mathcal{O}$, $~\exists u \in \mathcal{KO}$ such that $Ax + Bu + w \in \mathcal{O}, ~\forall w \in \mathcal{W}$.
\begin{assumption} \label{assum: Stage cost}
(Differentiable, Convex, and Positive Definite Stage Cost)
The stage cost $\ell(\cdot, \cdot)$ is differentiable and jointly convex in its arguments. Furthermore, we assume that $\ell(x, u) = 0, ~\forall x \in \mathcal{O}$ and $\forall u \in \mathcal{KO}$. Also, $\ell(x, u) > 0, ~\forall x \in  \mathcal{X} \setminus \mathcal{O}, ~\forall u \in  \mathcal{U} \setminus \mathcal{KO}$.
\end{assumption}

We want to compute a solution to the following infinite horizon OCP under the Assumptions \ref{assum: bounded noise}-\ref{assum: Stage cost} at the given $x_S$:
\begin{equation} \label{eq:ftocp}
\begin{aligned}
    & Q^{\star}(x_S) = \min_{\pi(\cdot)} ~\,\mathbb{E}_{w_{0:\infty}}\Biggr[\sum_{k=0}^\infty \ell(x_k, ~ \pi(x_k)) \Biggr] \\
    & \qquad \qquad ~ \textnormal{s.t.,} ~~ x_{k+1}  = A x_k + B \pi(x_k) + w_k, \\
    & \qquad \qquad  \qquad ~\, x_0 = x_S, ~ w_k \sim p(w), \\
    & \qquad \qquad  \qquad ~\, x_k \in \mathcal{X}, ~ \pi(x_k) \in \mathcal{U},  ~\forall w_k \in \mathcal{W},  \\
    & \qquad \qquad \qquad ~\, \forall k \geq 0.
\end{aligned}
\end{equation}
Note that system \eqref{eq:system} is uncertain and that the OCP \eqref{eq:ftocp} involves an optimization over state feedback policies $\pi(\cdot)$ \cite[Chap.1.2]{bertsekas2012dynamic}
\begin{remark}
Once the state $x_k^j$ of \eqref{eq:system} reaches the target set $\mathcal{O}$, it will remain within it. This is because remaining in the target set $\mathcal{O}$ is optimal according to Assumption \ref{assum: Stage cost}. 
Moreover, remaining in the target set $\mathcal{O}$ is feasible by Assumption \ref{assum: Robust Positive Invariance}.
\end{remark}

\section{Approach}
\subsection{Solution Approach}
There are three main challenges in solving \eqref{eq:ftocp}, namely: 
\begin{enumerate}[(C1)]
    \item Solving the problem \eqref{eq:ftocp} for $T \rightarrow \infty$ is computationally demanding.
    \item Minimizing the expected cost in \eqref{eq:ftocp} involves infinitely nested multivariate integrals.
    \item Optimizing over control policies \(\pi(\cdot)\) is an infinite dimensional optimization problem.
\end{enumerate}
\label{sub: Solution Approach to ftocp}
We address (C1)-(C3) by solving a simpler constrained OCP with prediction horizon $N=1$ in a receding horizon fashion and updating its parameters over multiple episodes.
Specifically, we design an MPC controller of the following form:
\begin{equation} \label{eq:MPC}
\begin{aligned}
    & {V}^{j} (x) = \min_{\substack{u_{k}}} ~\, \ell({x}_{k}, u_{k}) + \mathbb{E}_{w_k} \begin{bmatrix} Q^{j-1}({x}_{k+1}) \end{bmatrix}\\
    & \qquad \qquad \, \textnormal{s.t.,} \, ~ {x}_{k+1} = A {x}_{k} + B u_{k} + w_k, \\
    & \qquad \qquad \qquad \, {x}_{k} = {x}, ~ w_k \sim p(w)\\
    & \qquad \qquad \qquad \, {x}_{k} \in \mathcal{X}, ~ u_{k} \in \mathcal{U},\\
    & \qquad \qquad \qquad \, {x}_{k+1} \in \mathcal{SS}^{j-1}, ~ \forall w_k \in \mathcal{W},
\end{aligned}
\end{equation}
where ${x}_{k+i}$ is the state and ${u}_{k+i}$ is the input at predicted time step $k+i$. $\mathcal{SS}^{j-1}$ is the terminal set, and $Q^{j-1}(\cdot)$ is the terminal cost, which is the estimated value function defined on the terminal set $\mathcal{SS}^{j-1}$.
Let $T^j$ be the termination time of the $j$-th episode, i.e., $x^j_{T^j} \in \mathcal{O}$.

\begin{remark}
The correct notation in \eqref{eq:MPC} requires introducing $x_{0|k}=x$ and $x_{1|k}$ instead of $x_k$ and $x_{k+1}$ to denote the initial state and the first predicted state at time $k+1$. We chose to avoid this for readability.
\end{remark}
For time step $k$ of episode $j$, we store a state $x_k^j$ and the associated optimal input $u_k^\star = \pi^{j\star}(x_k^j)$ to learn the terminal set $\mathcal{SS}^{j}$ and the terminal cost $Q^{j}(\cdot)$ for the next episode $j+1$ after the episode $j$ is completed, which will be detailed in Sec. \ref{sec: learning terminal set and cost}.
$\mathcal{SS}^{j-1}$ is also called a safe set \cite{rosolia2017learning_journal} since if properly computed, it is robust positive invariant, i.e., $\forall x \in \mathcal{SS}^{j-1}, ~ \exists u \in \mathcal{U}, ~ \mathrm{s.t.,}~ Ax +Bu +w \in \mathcal{SS}^{j-1}, ~ \forall w \in \mathcal{W}$. 
Later in Sec. \ref{subsec: optimiality gap}, we will show that the estimated value function $Q^{j}(x_S)$ will converge the neighborhood of the optimal value function $Q^{\star}(x_S)$ defined in \eqref{eq:ftocp} as $j$ goes to infinity.
The block diagram of the proposed approach is illustrated in Fig. \ref{fig:blockdiagram}.
\begin{figure}[H]
\begin{center}
\includegraphics[width=0.80\linewidth,keepaspectratio]{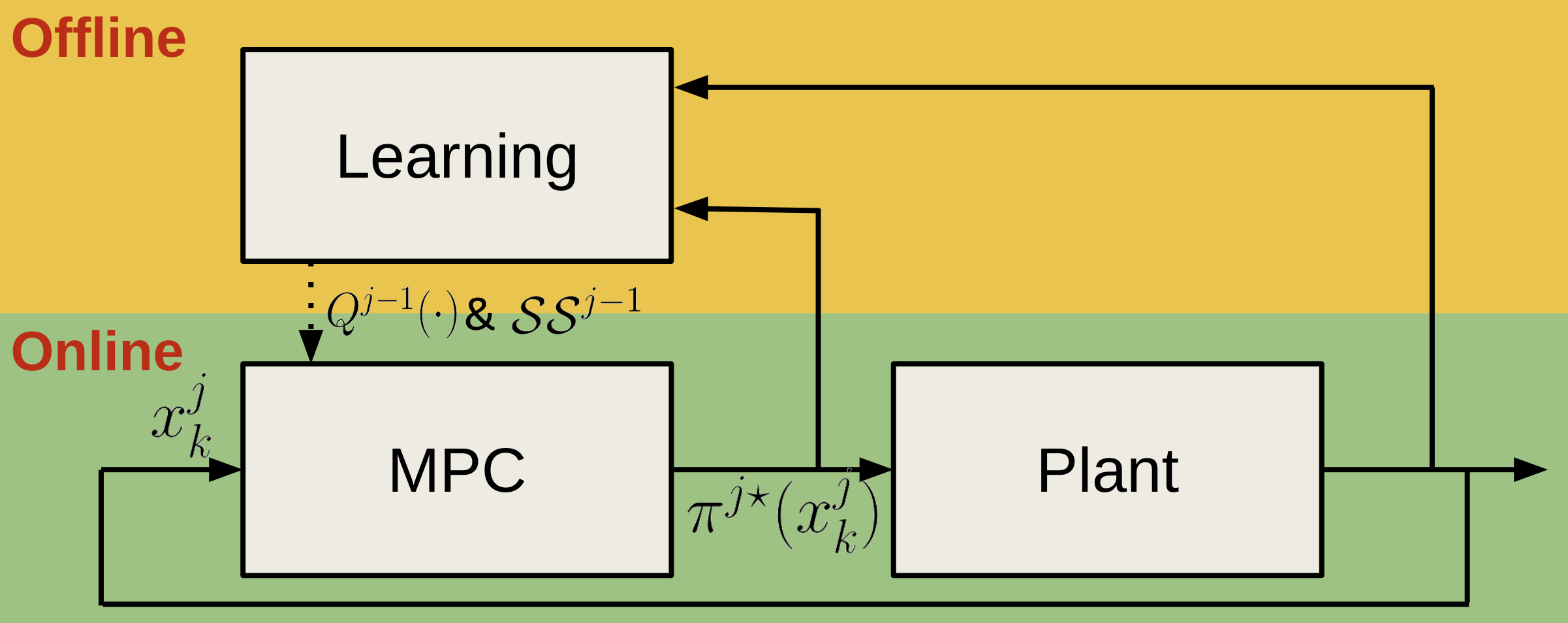}
\caption{Overall block diagram of the proposed approach. Online, we solve a tractable form of the MPC \eqref{eq:MPC}, which will be presented in \eqref{eq:MPC discrete}, and apply the optimal input to the system \eqref{eq:system}. Offline, we store closed-loop states and the associated inputs and update the terminal set $\mathcal{SS}^{j}$ and the terminal cost $Q^{j}(\cdot)$, which will be presented in Sec. \ref{sec: learning terminal set and cost}.}
\label{fig:blockdiagram}
\end{center}
\end{figure}
\vspace{-2.0em}
\begin{remark} \label{rmk: N = 1}
    We consider an OCP with $N=1$ because extending to a multi-step OCP ($N > 1$) involves challenging propagation of $p(w)$ along the predicted state trajectory unless it is Gaussian. However, with $N=1$, the terminal constraint of \eqref{eq:MPC}, ${x}_{k+1} \in \mathcal{SS}^{j-1}, ~ \forall w_k \in \mathcal{W}$, enforces that all closed-loop states $x_k^j$ of the system \eqref{eq:system} controlled by \eqref{eq:MPC} remain within $\mathcal{SS}^{j-1}$, and no data points fall outside this set. To explore beyond $\mathcal{SS}^{j-1}$, we incorporate exploration, which will be described in Sec. \ref{sec: exploration}, into the Learning phase in Fig. \ref{fig:blockdiagram}.
\end{remark}

\subsection{Tractable reformulation of \eqref{eq:MPC}} \label{sec: Approx}
\eqref{eq:MPC} is computationally intractable for real-time applications because evaluating the expected terminal cost in \eqref{eq:MPC} involves a multivariate integral. This section presents the tractable reformulation of \eqref{eq:MPC}.
\subsubsection{Tractable reformulation of the expected terminal cost in \eqref{eq:MPC}}
We reformulate the expected terminal cost in \eqref{eq:MPC} by approximating the multivariate integral with a discrete sum by discretizing the disturbance set $\mathcal{W}$.
Our discretization approach involves sampling disturbances randomly within the disturbance set $\mathcal{W}$ and including the vertices of $\mathcal{W}$. 
The discretized disturbances consist of all sampled points and the vertices of $\mathcal{W}$ as shown in Fig. \ref{fig:disturbance set grid}.
\begin{figure}[ht]
\begin{center}
\includegraphics[width=0.35\linewidth,keepaspectratio]{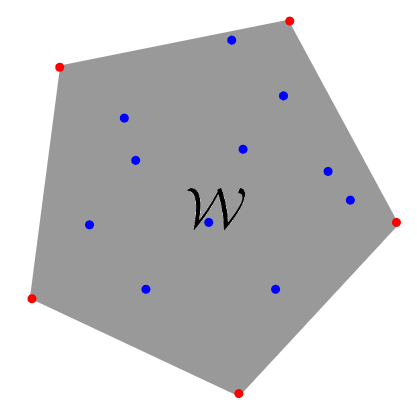}
\caption{A two-dimensional example of the discretization method. The dark gray polytope is the disturbance set $\mathcal{W}$. All dots are the discretized disturbances. The blue dots denote randomly sampled points while the red dots denote vertices of the disturbance set $\mathcal{W}$.}
\label{fig:disturbance set grid}
\end{center}
\vspace{-1.0em}
\end{figure}\\
Let $M$ denote the number of discretized disturbances and $\tilde{w}_m$ denote $m$-th discretized disturbance.
As the sampled disturbance set $\{\tilde{w}_m\}_{m=1}^M$ includes vertices of $\mathcal{W}$, any disturbance $w \in \mathcal{W}$ can be parameterized as a convex combination of the discretized disturbances as follows:
\begin{equation} \label{eq: cvx sum with mu w}
\begin{aligned}
    & w = \sum_{m=1}^M \mu_m(w) \tilde{w}_m,
\end{aligned}
\end{equation}
where $\mu_m(w)$ is the parameterized convex coefficient.
We will use \eqref{eq: cvx sum with mu w} to replace a generic disturbance $w$ with the convex combination of the sampled disturbance set $\{\tilde{w}_m\}_{m=1}^M$.
This replacement is needed to approximate the expected cost.
To compute the optimality gap in Section \ref{subsec: optimiality gap}, the parameterized coefficients need to satisfy the property introduced next.
\begin{assumption} \label{assum: muw assumption}
$\forall w \in \mathcal{W}$, let $i(w) \in \{m | m \in \{1,\cdots,M\} \text{ and } \mu_m(w) \neq 0\}$, the following conditions hold:
\ref{assum: muw assumption}.1. Existence of bound $\delta(w,M)$
\begin{equation} \label{eq: delta_bnd}
    \exists \delta(w,M)>0, ~||w - \tilde{w}_{i(w)}|| \leq \delta(w,M). 
\end{equation}
\ref{assum: muw assumption}.2 Decreasing bound $\delta(w,M)$ as increasing $M$ \\
Consider $M=M_1$, and we continuously sample the disturbance within $\mathcal{W}$ while maintaining the previous $M_1$ samples.
Then, we have the following:
\begin{equation} \label{eq: M1 M2 assumption}
    \exists N > M_1, ~\forall M_2 \geq N, ~\delta(w,M_1) > \delta(w,M_2).
\end{equation}
\end{assumption}

Appendix \ref{sec: calculating muw} describes different ways to calculate $\mu_m(w)$, that satisfy Assumption~\ref{assum: muw assumption}, and its expected value for $p(w)$, denoted $p_m$ next. Note that \(\mu_m(w)\) is not necessarily unique for a given \(w\). In such cases, any \(\mu_m(w)\), such as the least norm solution, can be chosen.

\begin{remark}
In Assumption~\ref{assum: muw assumption}.1, the function $\delta(w, M)$ bounds the distance between the actual disturbance $w$ and the discretized disturbance $\tilde{w}_{i(w)}$, where $\mu_{i(w)}(w) \neq 0$, ensuring $||w - \tilde{w}_{i(w)}|| \leq \delta(w, M)$. This bound represents the diameter of the convex hull of discretized disturbances with non-zero $\mu_{i(w)}(w)$. 
In Assumption~\ref{assum: muw assumption}.2 as the number of discretized disturbances $M$ increases, $\delta(w, M)$ decreases. 
\end{remark}

Once the function $\mu_m(w)$ is calculated, we take the expectation of it with respect to $p(w)$, which is a coefficient of the discrete sum to approximate the expected terminal cost in \eqref{eq:MPC}:
\begin{equation} \label{eq: pm}
\begin{aligned}
    & p_m = \mathbb{E}_w [\mu_m(w)].
\end{aligned}
\end{equation}
Note that $p_m$ is computed prior to controller execution, and once computed, its value will not change for the entire process.

Using the coefficients $p_m$, we approximate the expected terminal cost in the \eqref{eq:MPC} with a discrete sum as follows:
\begin{equation} \label{eq: terminal cost approximation}
    \begin{aligned}
        & \mathbb{E}_{w_k} \begin{bmatrix} Q^{j-1}({x}_{k+1}) \end{bmatrix} \simeq \sum_{m=1}^{M} p_m Q^{j-1}({x}_{k+1}) 
    \end{aligned}
\end{equation}
\subsubsection{Tractable MPC problem}
The tractable reformulation of the MPC problem \eqref{eq:MPC} can be achieved by replacing the terminal cost of \eqref{eq:MPC} with its approximation in \eqref{eq: terminal cost approximation} and robustifying the terminal constraint ${x}_{1|k}^j \in \mathcal{SS}^{j-1}, \, \forall w_k^j \in \mathcal{W}$.  
In summary, the tractable reformulation of \eqref{eq:MPC} is given as follows:
\begin{equation} \label{eq:MPC discrete}
\begin{split}
    & \hat{V}^{j} (x) = \min_{\substack{u_{k}}} ~\,\ell({x}_{k}, u_{k}) + \sum_{m=1}^{M} p_m Q^{j-1}(\bar{x}_{k+1} + \tilde{w}_m)\\
    & \qquad \qquad \, \textnormal{s.t.,} ~ \, \bar{x}_{k+1} = A {x}_{k} + B u_{k}, \\
    & \qquad \qquad \qquad \, {x}_{k} = {x}, \\
    & \qquad \qquad \qquad \, {x}_{k} \in \mathcal{X}, ~ u_{k} \in \mathcal{U}, \\
    & \qquad \qquad \qquad \, \bar{x}_{k+1} \in \mathcal{SS}^{j-1} \ominus \mathcal{W}, 
\end{split}
\end{equation}
where $\bar{x}_{k+1}$ is a nominal state at predicted time step $k+1$.

After solving \eqref{eq:MPC discrete}, we apply the optimal MPC policy $\pi^j(\cdot)$:
\begin{equation} \label{eq:MPC policy}
    u_k^j = \pi^{j \star}(x_k^j) = u_{k}^{\star}
\end{equation}
to system \eqref{eq:system} at time $k$ of the episode $j$ in closed-loop.
Note that \eqref{eq:MPC discrete} and \eqref{eq:MPC policy} correspond to the "MPC" block in Fig. \ref{fig:blockdiagram}.

\subsubsection{Bounding the Expected Terminal Cost: Proof of Upper Bound}
In this section, we show that the approximation in \eqref{eq: terminal cost approximation} is an upper bound of the expected terminal cost in \eqref{eq:MPC}, as demonstrated in  \cite[Thm.1]{de2005stochastic}.
Later in Sec. \ref{sec: value function update}, we will define the terminal cost function $Q^j(\cdot)$ and show that it is convex. For now, we assume that $Q^{j}(\cdot)$ is convex.

\begin{proposition} \label{prop: lyapunov function of the closed-loop system}
Suppose the value function $Q^{j}(\cdot)$ is a convex function.
Let Assumptions \ref{assum: bounded noise}-\ref{assum: Stage cost} hold.
Then, the value function $Q^{j}(\cdot)$ satisfies the following inequality for all $x \in \mathcal{SS}^{j}$:
\begin{equation} \label{eq: upperbound}
\begin{aligned}
    & \sum_{m=1}^{M} p_m Q^{j}(Ax + Bu + \tilde{w}_m) \geq \mathbb{E}_w[Q^{j}(Ax + Bu + w)]. \\
\end{aligned}
\end{equation}
\end{proposition}
\begin{proof}
Remind that $\mu_m(w)$ is a convex coefficient and $w = \sum_{m=1}^M \mu_m(w) \tilde{w}_m$ from \eqref{eq: cvx sum with mu w}.
Thus, we have the following:
\begin{equation*}
\begin{aligned}
    & ~~~~ \mathbb{E}_w[Q^{j}(Ax + Bu + w)] \\
    & = {\smallint}_{w\in\mathcal{W}} Q^{j}(Ax + Bu + w) p(w) dw \\
    & = {\smallint}_{w\in\mathcal{W}} Q^{j}\big({\scriptstyle\sum}_{m=1}^{M} \mu_m(w)(Ax + Bu + \tilde{w}_m)\big) p(w) dw
\end{aligned}
\end{equation*}
The first equality is the definition of expectation.
Due to the convexity of the value function $Q^j(\cdot)$, we have the following
\begin{equation*}
\begin{aligned}
    & ~~~~ {\smallint}_{w\in\mathcal{W}} Q^{j}\big({\scriptstyle\sum}_{m=1}^{M} \mu_m(w)(Ax + Bu + \tilde{w}_m)\big) p(w) dw\\
    & \leq {\smallint}_{w\in\mathcal{W}} {\scriptstyle\sum}_{m=1}^{M} \mu_m(w) Q^{j}(Ax + Bu + \tilde{w}_m) p(w) dw \\
    & = {\scriptstyle\sum}_{m=1}^{M} Q^{j}(Ax + Bu + \tilde{w}_m) {\smallint}_{w\in\mathcal{W}} \mu_m(w)  p(w) dw \\
    & = {\scriptstyle\sum}_{m=1}^{M} p_m Q^{j}(Ax + Bu + \tilde{w}_m),
\end{aligned}
\end{equation*}
where the last equality is from the definition of $p_m$ in \eqref{eq: pm}.
\end{proof}

In Section~\ref{subsec: Convergence}, we show that property \eqref{eq: upperbound} is required to prove the stochastic stability of the closed loop system \cite{de2005stochastic}.

\section{Learning the safe Set $\mathcal{SS}^j$ \\ and Value Function $Q^j(\cdot)$ using Data} \label{sec: learning terminal set and cost}
In this section, we present how the safe set $\mathcal{SS}^j$ and the value function $Q^j(\cdot)$ in \eqref{eq:MPC discrete} are updated using the data, which is the "Learning" block in Fig. \ref{fig:blockdiagram}.
This section is organized as follows.
First, we describe how to initialize the safe set $\mathcal{SS}^j$ and the value function $Q^j(\cdot)$, i.e., how to compute $\mathcal{SS}^0$ and $Q^0(\cdot)$.
Second, we present \textit{Exploration}, a new method to explore the state space while satisfying the constraints \eqref{eq:constraints}. This is the novel part and one of our contributions.
Third, we describe how to update $\mathcal{SS}^j$ and $Q^j(\cdot)$ offline using the data.
This part is similar to the methods in \cite[Sec.V.B-D]{rosolia2021robust}.
At the end of this section, we introduce the properties of the safe set $\mathcal{SS}^j$ and the value function $Q^j(\cdot)$.

\subsection{Initialization} \label{sec: initialization}
In this subsection, we describe how to compute $\mathcal{SS}^0$ and $Q^0(\cdot)$. 
To do that, we make the following assumption on the existence of a feasible solution to the problem \eqref{eq:ftocp}.
\begin{assumption} \label{assum: initial solution}
    A sub-optimal controller of the infinite horizon OCP \eqref{eq:ftocp} is available, and the system \eqref{eq:system} controlled by this sub-optimal controller ends the episode within a finite time.
\end{assumption}
\begin{remark}
    Assumption \ref{assum: initial solution} is not overly restrictive in some practical applications. For instance, in autonomous racing scenarios, which are formulated as an infinite horizon OCP in \cite{rosolia2019racing,rosolia2017learning_journal}, a tube MPC in \cite{wischnewski2022tube}, which follows a planned trajectory, can be utilized for a sub-optimal controller. 
\end{remark}

We provide a concise summary of the design process for $\mathcal{SS}^0$ and $Q^0(\cdot)$ when the sub-optimal controller satisfied Assumption \ref{assum: initial solution} is a Tube MPC \cite{mayne2005robusto}. Please refer to Appendix \ref{sec: tube MPC initialize} for a more comprehensive understanding.

We conduct one closed-loop episode of the system \eqref{eq:system} controlled by the tube MPC, during which we collect the closed-loop nominal state and the surrounding tube.
Let $\bar{x}_{0:T}^{0\star}$ denote the collected closed-loop states, and let $\mathcal{E}$ denote the corresponding tube.
Then we construct the following sets:
\begin{equation} \label{eq: tubes}
    \begin{split}
        & \mathcal{E}_k = \bar{x}_{k}^{0\star} \oplus \mathcal{E}, ~ k \in \{0,\cdots,T+1\},
    \end{split}
\end{equation}
where $\mathcal{E}_{0:T+1}$ represents the trajectory of tubes.
Since $\mathcal{E}$ is time invariant, the number of vertices of $\mathcal{E}_k$ is same for all $k \in {0,\cdots,T+1}$. Let $l$ denote the number of vertices of $\mathcal{E}_k$ and $v_1^{\mathcal{E}_k},\cdots, v_{l}^{\mathcal{E}_k}$ denotes the vertices of $\mathcal{E}_k$.

We construct the initial state data matrix $\mathbf{X}^{0}$ to store the collected data in a matrix form as follows:
\begin{equation*} 
\begin{split}
    & \mathbf{X}^{0} = \Big[v_1^\mathcal{O},\cdots, v_{l_\mathcal{O}}^\mathcal{O}, v_1^{\mathcal{E}_0},\cdots, v_{l}^{\mathcal{E}_0}, \cdots, v_{1}^{\mathcal{E}_{T+1}},\cdots, v_{l}^{\mathcal{E}_{T+1}}\Big],
\end{split}
\end{equation*}
which consists of the vertices of the target set $\mathcal{O}$ and $\mathcal{E}_k(\bar{x}_k^{0\star}), ~\forall k \in \{0,\cdots,T+1\}$.

The initial safe set $\mathcal{SS}^0$ is then defined as a convex hull of all column vectors in the initial state data matrix $\mathbf{X}^{0}$, as follows:
\begin{equation} \label{eq: initial terminal constraint app}
    \mathcal{SS}^0 = \mathrm{conv}\Biggl\{\bigcup_{i = 1}^{L} \{[\mathbf{X}^{0}]_i\}\Biggr\},
\end{equation}
where $L$ is the number of the states stored in the state data matrix $\mathbf{X}^{0}$, and $[\mathbf{X}^{0}]_i$ denotes $i$-th column vector of $\mathbf{X}^{0}$.
Thus, $\mathcal{SS}^0$ is the convex hull of the target set $\mathcal{O}$ and the trajectory of tubes $\mathcal{E}_{0:T+1}$.

Similarly, during the episode, we collect the corresponding control input of each column vector of $\mathbf{X}^{0}$.
Let $\pi_\mathrm{Tube}$ denote the tube MPC policy.
Then, $i$-th column vector of the initial input data matrix $\mathbf{U}^{0}$ is defined as follows:
\begin{equation}
    [\mathbf{U}^{0}]_i = \pi_\mathrm{Tube}([\mathbf{X}^{0}]_i).
\end{equation}

For each state in $\mathbf{X}^{0}$ and the corresponding input in $\mathbf{U}^{0}$, we calculate an estimated cost-to-go value and store it in a cost data matrix $\mathbf{J}^{0}$.
$i$-th element of the cost data matrix $\mathbf{J}^{0}$ should satisfy the following conditions:
\begin{equation} \label{eq: Q0 condition}
    \begin{aligned}
        & [\mathbf{J}^{0}]_i  \geq \ell([\mathbf{X}^{0}]_i,[\mathbf{U}^{0}]_i) \\
        & ~~~~~~~~~~ + \sum_{m=1}^M p_m Q^0(A[\mathbf{X}^{0}]_i + B[\mathbf{U}^{0}]_i + \tilde{w}_m)
    \end{aligned}
\end{equation}
where $Q^0(\cdot)$ is the initial value function defined on $\mathcal{SS}^{0}$ computed using $\mathbf{X}^0$ and $\mathbf{J}^{0}$ as:
\begin{equation} \label{eq: initial value function}
\begin{split}
    & Q^0(x) = \min_{\bm{\lambda}^0} ~\mathbf{J}^0 \bm{\lambda}^0 \\
    & \qquad ~~~~~~ \, \textnormal{s.t.,} ~ \mathbf{X}^0 \bm{\lambda}^0 = x, \\
    & \qquad \qquad ~~~~~~ \bm{\lambda}^0 \geq \bm{0}, \,\, \bm{1}^\top \bm{\lambda}^0 = 1.
\end{split}
\end{equation}
$\bm{\lambda}^{0}$ is a convex coefficient vector of the problem \eqref{eq: initial value function}, which represents $x \in \mathcal{SS}^{0}$ as a convex combination of the data points of $\mathbf{X}^0$ as described in the first constraint of \eqref{eq: initial value function}. 
In Appendix \ref{sec: tube MPC initialize}, we present a way to calculate $\mathbf{J}^{0}$ to satisfy the condition \eqref{eq: Q0 condition}.
Note that $Q^0(x)$ is a convex, continuous, piecewise affine function as it is an optimal objective function of the multi-parametric linear program \cite{borrelli2017predictive}. When calculating the optimal input in \eqref{eq:MPC discrete} at episode $j=1$, the function $Q^0(x)$ can be incorporated into the problem \eqref{eq:MPC discrete} and form a single unified optimization problem. 


\subsection{Exploration} \label{sec: exploration}
Exploration refers to finding a subset of states from the set $\{x \in \mathcal{X} \mid \eqref{eq:MPC discrete} \text{ is feasible}\}$ after episode $j$ has finished and before episode $j+1$ begins. This process consists of two steps: First, we compute a descent direction to improve performance through exploration. Second, we explore states by identifying those outside the safe set $\mathcal{SS}^j$ in the descent direction while examining a subset of states within $\mathcal{SS}^j$. 

The condition that "\eqref{eq:MPC discrete} \text{ is feasible}" means that the state $x$ can be robustly guided to the safe set $\mathcal{SS}^{j-1}$. This condition is crucial for proving the robust positive invariance of the updated set $\mathcal{SS}^{j}$ in Proposition \ref{prop: SS0 is rpi}, which is key to ensuring the recursive feasibility of \eqref{eq:MPC discrete} in Sec. \ref{subsec: feasibility}. 
The set of explored states is then added to $\mathcal{SS}^{j-1}$ and used to update the function $Q^{j-1}(\cdot)$ to improve the expected cost in \eqref{eq:ftocp} for the closed-loop system \eqref{eq:system}, which is controlled by the MPC policy \eqref{eq:MPC policy}, referred to as the controller's "performance."

\begin{remark}
In prior LMPC works \cite{rosolia2017learning_journal, rosolia2018stochastic, rosolia2021robust}, exploration beyond the safe set was naturally embedded into the formulation. In those approaches, the horizon length is $N > 1$, which allows for some visited states $x_k^j$ to fall outside $\mathcal{SS}^{j-1}$. This differs from our setting, where $N = 1$, as noted in Remark \ref{rmk: N = 1}.
\end{remark}

As in \cite[Algorithm 2]{bujarbaruah2022robust}, the exploration can be carried out by randomly selecting a state $x$ from the set $\mathcal{X}$ and checking whether the state $x$ is feasible for the MPC \eqref{eq:MPC discrete}. 
However, as this method selects the directions of exploration not necessarily relevant to the direction of improving the performance, it can be computationally inefficient.

\subsubsection{Calculating a descent direction} 
To iteratively reduce the cost, it is crucial to determine a descent direction that aligns with the gradient of the objective function in \eqref{eq:MPC discrete} at its optimal solution $u_{k}^{\star}$. The descent direction is a vector in the input space that points toward a lower-cost region. By following this direction, the algorithm minimizes the cost function at each step in a greedy manner, avoiding inefficient randomization.

For brevity, we reformulate the problem \eqref{eq:MPC discrete} as follows:
\begin{subequations} \label{eq:MPC reformulated}
\begin{align}
    & \min_{u} ~ f_k^j(u)\\
    & ~ \textnormal{s.t.,} ~~ H (A x_k^j + B u) \leq h, \label{subeq: terminal constraints reform MPC} \\
    & \qquad ~ H_u u \leq h_u, \label{subeq: input constraints reform MPC}
\end{align}
\end{subequations}
where $f_k^j(u)= \ell(x_k^j, u) + \sum_{m=1}^{M} p_m Q^{j-1}(Ax_k^j + Bu + \tilde{w}_m)$, which is a cost function of \eqref{eq:MPC discrete} in abstract form, and $\mathcal{SS}^{j-1} \ominus \mathcal{W} = \{x | H x \leq h\}$.
We also replaced $x$ in (11) with $x_k^j$ to clarify that the initial state depends on the time index $k$ and the iteration $j$.
Note that, $\mathcal{SS}^{j-1} \ominus \mathcal{W}$ is a convex polytope that will be introduced in Sec. \ref{sec: value function update}.
Although the parameters $H$ and $h$ in \eqref{subeq: terminal constraints reform MPC} are updated for every episode $j$, we denote the constraint without any superscript $^j$ for brevity.
The goal is to calculate the descent direction using the Karush-Kuhn-Tucker (KKT) conditions \cite{boyd2004convex}.

The optimal solutions of \eqref{eq:MPC reformulated}, ${u}^{\star}$, satisfies the following stationarity condition of the KKT conditions:
\begin{subequations} \label{eq:KKT}
\begin{align}
    & \nabla f_k^j({u}^{\star}) + (H B)^\top \bm{\nu}^{j\star}_{k} + H_u^\top \bm{\gamma}^{j\star}_{k}= \bm{0}, \label{subeq: station}
\end{align}
\end{subequations}
where $\bm{\nu}^{j\star}_{k} \geq \bm{0}$ and $\bm{\gamma}^{j\star}_{k} \geq \bm{0}$ are the optimal dual variables at time step $k$ of the episode $j$ associated with the constraints \eqref{subeq: terminal constraints reform MPC} and \eqref{subeq: input constraints reform MPC}, respectively. 
From the condition \eqref{subeq: station}, we can calculate the gradient of the cost function $\nabla f_k^j({u}^{\star})$ as follows:
\begin{equation} \label{eq: descent direction}
    - \nabla f_k^j({u}^{\star}) = (H B)^\top \bm{\nu}^{j\star}_{k} + H_u^\top \bm{\gamma}^{j\star}_{k}
\end{equation}
Using $\nabla f_k^j({u}^{\star})$, the descent direction $\bm{d}_{\mathrm{exp}} \in \mathbb{R}^{n_u}$ is a vector which satisfies the following inequality:
\begin{equation} \label{eq: descent direction vector d}
    \nabla f_k^j({u}^{\star})^\top \bm{d}_{\mathrm{exp}} \leq 0.
\end{equation}
Note that $\bm{d}_{\mathrm{exp}}\in \mathbb{R}^{n_u}$ is the vector in the input space.

\begin{remark}
    KKT conditions\cite{boyd2004convex} require the differentiability of the cost function. While Assumption \ref{assum: Stage cost} guarantees the differentiability of the stage cost $\ell(\cdot, \cdot)$, the function $Q^j(\cdot)$ is continuous and piecewise affine, which will be introduced in Sec. \ref{sec: value function update}, but lacks differentiability at certain points. However, these non-differentiable points have zero probability of occurring in states. Hence, this paper primarily uses gradient notation instead of subgradient notation for the cost function in \eqref{eq:MPC reformulated}.
    If the cost function $f_k^j({u})$ becomes non-differentiable, its subgradient must be used.
\end{remark}

\subsubsection{Finding states to be explored} We find states to be explored both outside and within the safe set $\mathcal{SS}^{j-1}$.
Our exploration strategy outside $\mathcal{SS}^{j-1}$ is as follows: for each collected state $x_k^j$, we find a 1-step robust reachable set(Definition in Appendix \ref{sec: robust reachable set}) from $x_k^j$ that is situated at the maximum distance from the current safe set $\mathcal{SS}^{j-1}$ in the descent direction $\bm{d}_{\mathrm{exp}}$. 

To ensure that all states in a 1-step robust reachable set are feasible solutions of \eqref{eq:MPC discrete}, we impose the following constraint: all states in the 1-step robust reachable set should be robustly steered to $\mathcal{SS}^{j-1}$. 
In particular, we try to solve the following optimization problem where the decision variables are the descent direction vector $\bm{d}_{\mathrm{exp}} \in \mathbb{R}^{n_u}$ and a control policy $\pi_{\mathrm{exp}}\colon \mathcal{X} \to \mathcal{U}$ as follows:
\begin{subequations} \label{eq: exploration optimization ideal}
    \begin{align}
    & d_{\mathcal{SS}^j}(x_k^j, {u}^{\star}) = \\
    & \max_{\substack{\bm{d}_{\mathrm{exp}}, \\ \pi_{\mathrm{exp}}(\cdot)}} ~\mathrm{dist}({x}_{k,\mathrm{exp}}^j, \mathcal{SS}^{j-1})  \label{subeq: distance to the target set}\\
    &  ~~~ \textnormal{s.t.,} ~~ \nabla f_k^j({u}^{\star})^\top {\bm{d}_{\mathrm{exp}}} \leq 0 \label{subeq: descent direction} \\
    & \quad \qquad \, {x}_{k,\mathrm{exp}}^j = A x_k^j + B ({u}^{\star} + \bm{d}_{\mathrm{exp}}) + w, ~w \sim p(w), \label{subeq: uncertain dynamics of exploring point} \\
    & \quad \qquad \,  {x}_{k,\mathrm{exp}}^j \in \mathcal{X}, ~  \forall w \in \mathcal{W}, \label{subeq: state constraints exploring point} \\
    & \quad \qquad \,  {u}^{\star} + \bm{d}_{\mathrm{exp}} \in \mathcal{U}, \label{subeq: input constraints exploring point} \\
    & \quad \qquad \, A {x}_{k,\mathrm{exp}}^j + B \pi_{\mathrm{exp}}({x}_{k,\mathrm{exp}}^j) \in \mathcal{SS}^{j-1}, ~  \forall w \in \mathcal{W}, \label{subeq: terminal constraint} \\
    & \quad \qquad \, \pi_{\mathrm{exp}}({x}_{k,\mathrm{exp}}^j) \in \mathcal{U}, ~  \forall w \in \mathcal{W}, \label{subeq: input policy satisfies input constraint}
    \end{align}
\end{subequations}
where ${x}_{k,\mathrm{exp}}^j$ is a state to be explored.
The cost function \eqref{subeq: distance to the target set} is a distance measure between the state ${x}_{k,\mathrm{exp}}^j$ and the safe set $\mathcal{SS}^{j-1}$.
The constraint \eqref{subeq: descent direction} ensures that the vector $\bm{d}_{\mathrm{exp}}$ is a descent direction.
The set of feasible solutions ${x}_{k,\mathrm{exp}}^j$ for all $ w \in \mathcal{W}$ is 1-step robust reachable set because of \eqref{subeq: uncertain dynamics of exploring point}, \eqref{subeq: state constraints exploring point}, and \eqref{subeq: input constraints exploring point}.
The constraints \eqref{subeq: terminal constraint} and \eqref{subeq: input policy satisfies input constraint} are to robustly steer the 1-step robust reachable set to $\mathcal{SS}^{j-1}$. 

There are two main challenges in solving the exploration optimization problem \eqref{eq: exploration optimization ideal}.
First, optimizing over control policies $\pi_\mathrm{exp}(\cdot)$ involves an infinite-dimensional optimization.
Second, maximizing a distance measure is non-convex.

We address the \textbf{first challenge} by approximating the control policy with a convex combination of the finite number of inputs, denoted as ${u}_{i}^\mathcal{W}$, where $i\in \{1,\cdots,l_w\}$.
Remind that $l_w$ is the number of vertices of $\mathcal{W}$ and  $v_i^\mathcal{W}$ is the $i$-th vertex of $\mathcal{W}$.
Each control input ${u}_{i}^\mathcal{W}$ satisfies the following:
\begin{subequations} \label{eq: convex control policy condition}
    \begin{align}
        & {u}_{i}^\mathcal{W} \in \mathcal{U}, \label{subeq: input policy satisfies input constraint hat} \\
        & A {x}_{i}^\mathcal{W} + B {u}_{i}^\mathcal{W} + w \in \mathcal{SS}^{j-1}, ~  \forall w \in \mathcal{W}, \label{subeq: terminal constraint hat} \\
        & {x}_{i}^\mathcal{W} = A x_k^j + B ({u}^{\star} + \bm{d}_{\mathrm{exp}}) + v_i^\mathcal{W}. 
    \end{align}
\end{subequations}
This means that a control input ${u}_{i}^\mathcal{W}$ robustly steers back the system \eqref{eq:system} at ${x}_{i}^\mathcal{W}$ to the safe set $\mathcal{SS}^{j-1}$, where ${x}_{i}^\mathcal{W}$ is a predicted state of the system \eqref{eq:system} starting from $x_k^j$, controlled by ${u}^{\star} + \bm{d}_{\mathrm{exp}}$ and perturbed by $v_i^\mathcal{W}$. Later in Corollary \ref{cor: exp feasible}, it will be shown that there always exist ${u}_{i}^\mathcal{W}$ satisfying \eqref{eq: convex control policy condition} if the MPC \eqref{eq:MPC discrete} is feasible at the state $x_k^j$.

The control policy is parameterized as a convex combination of  ${u}_{i}^\mathcal{W}, ~i\in \{1,\cdots,l_w\}$ as follows:
\begin{equation} \label{eq: convex control policy}
    \begin{aligned}
    & \hat{\pi}_{\mathrm{exp}}(x) \in \{u | u = [{u}_{1}^\mathcal{W},\cdots,{u}_{l_w}^\mathcal{W}] \bm{\lambda}, ~\bm{\lambda} \geq 0, ~\bm{1}^\top \bm{\lambda} = 1, \\
    & \qquad \qquad \quad ~~ x = A x_k^j + B ({u}^{\star} + \bm{d}_{\mathrm{exp}}) + [v_1^\mathcal{W},\cdots,v_{l_w}^\mathcal{W}] \bm{\lambda} \}.
    \end{aligned}
\end{equation}
Now, we prove that \eqref{eq: convex control policy} is a control policy that satisfies the constraints \eqref{subeq: terminal constraint} and \eqref{subeq: input policy satisfies input constraint}.
Later in Sec. \ref{sec: value function update}, we will define the safe set $\mathcal{SS}^{j-1}$ and show that it is convex. For now, we assume that $\mathcal{SS}^{j-1}$ is convex.
\begin{proposition} \label{prop: convex policy is valid}
(The control policy \eqref{eq: convex control policy} is a valid policy.) Suppose that $\mathcal{SS}^{j-1}$ is a convex set. Let Assumption \ref{assum: bounded noise} hold.
Then, the control policy \eqref{eq: convex control policy} satisfies \eqref{subeq: terminal constraint} and \eqref{subeq: input policy satisfies input constraint}.
\end{proposition}
\begin{proof}
Note that ${x}_{i}^\mathcal{W}$, where $i\in \{1,\cdots,l_w\}$, are the predicted states perturbed by $i$-th vertex of $\mathcal{W}$, $v_i^\mathcal{W}$, as described in \eqref{eq: convex control policy condition}.
Since $\mathcal{W}$ is a convex set, we can write the state $x_{k,\mathrm{exp}}^j$ defined in \eqref{subeq: uncertain dynamics of exploring point} for all $w \in \mathcal{W}$ as a convex combination of ${x}_{i}^\mathcal{W}$, where $i\in \{1,\cdots,l_w\}$, as follows:
\begin{equation} \label{eq: x_exp in cvx comb}
    \begin{aligned}
        & x_{k,\mathrm{exp}}^j ={\scriptstyle\sum}_{i=1}^{l_w} [\bm{\lambda}]_i {x}_{i}^\mathcal{W}\\
        & \qquad ~ \, =  A x_k^j + B ({u}^{\star} + \bm{d}_{\mathrm{exp}}) + [v_1^\mathcal{W},\cdots,v_{l_w}^\mathcal{W}] \bm{\lambda},
    \end{aligned}
\end{equation}
where $\bm{\lambda}$ is an associated convex coefficient vector such that $\bm{\lambda} \geq 0, ~\bm{1}^\top \bm{\lambda} = 1$, and $[\bm{\lambda}]_i$ is the $i$-th element of the vector $\bm{\lambda}$.
From ${u}_{i}^\mathcal{W} \in \mathcal{U}$ in \eqref{eq: convex control policy condition}, $\hat{\pi}_{\mathrm{exp}}(\cdot)$ in \eqref{eq: convex control policy}, and \eqref{eq: x_exp in cvx comb}, we have that
\begin{equation}
    \begin{aligned}
        & \hat{\pi}_{\mathrm{exp}}(x_{k,\mathrm{exp}}^j) = {\scriptstyle\sum}_{i=1}^{l_w} [\bm{\lambda}]_i {u}_{i}^\mathcal{W} = [{u}_{1}^\mathcal{W},\cdots,{u}_{l_w}^\mathcal{W}] \bm{\lambda}.
    \end{aligned}
\end{equation}
This implies that $\hat{\pi}_{\mathrm{exp}}(x_{k,\mathrm{exp}}^j)$ can be written as a convex combination of ${u}_{i}^\mathcal{W}, ~i\in\{1,\cdots,l_w\}$.
As $\mathcal{U}$ is a convex set, the input constraint \eqref{subeq: input policy satisfies input constraint} is satisfied,i.e., $\hat{\pi}_{\mathrm{exp}}(x_{k,\mathrm{exp}}^j) \in \mathcal{U}$.
Furthermore, from $A {x}_{i}^\mathcal{W} + B {u}_{i}^\mathcal{W} + w \in \mathcal{SS}^{j-1}, ~  \forall w \in \mathcal{W}$ in \eqref{eq: convex control policy condition} and by convexity of the set $\mathcal{SS}^{j-1}$, we have that
\begin{equation}
    \begin{aligned}
        & ~~~ {\scriptstyle\sum}_{i=1}^{l_w} [\bm{\lambda}]_i (A {x}_{i}^\mathcal{W} + B {u}_{i}^\mathcal{W} + w) \\
        & = Ax_{k,\mathrm{exp}}^j +B\pi_{\mathrm{exp}}(x_{k,\mathrm{exp}}^j) + w  \in \mathcal{SS}^{j-1}, ~\forall w \in \mathcal{W},
    \end{aligned}
\end{equation}
which shows that $\pi_{\mathrm{exp}}(x_{k,\mathrm{exp}}^j)$ \eqref{eq: convex control policy} satisfies \eqref{subeq: terminal constraint}.
\end{proof}

\begin{remark} \label{rmk : replacing the control policy}
We replace the original control policy \(\pi_{\mathrm{exp}}(\cdot)\) with the parameterized policy \(\hat{\pi}_{\mathrm{exp}}(\cdot)\) \eqref{eq: convex control policy}, within the optimization problem \eqref{eq: exploration optimization ideal}. To compute \(\hat{\pi}_{\mathrm{exp}}(\cdot)\), we should solve for (a) the control inputs \({u}_{i}^\mathcal{W}\), \(i \in \{1, \dots, l_w\}\), that satisfy \eqref{eq: convex control policy condition}, and (b) the convex coefficients \(\bm{\lambda}\) that parameterize \(\hat{\pi}_{\mathrm{exp}}(\cdot)\) in \eqref{eq: convex control policy}.
However, solving for \(\bm{\lambda}\) is unnecessary. Since \(\mathcal{SS}^{j-1}\) and \(\mathcal{U}\) are convex, any convex combination of \eqref{subeq: terminal constraint hat} and \eqref{subeq: input policy satisfies input constraint hat} will automatically satisfy \eqref{subeq: terminal constraint} and \eqref{subeq: input policy satisfies input constraint hat}, respectively.
Thus, in the tractable reformulation of \eqref{eq: exploration optimization ideal}, \eqref{subeq: terminal constraint} and \eqref{subeq: input policy satisfies input constraint} are replaced with \eqref{eq: convex control policy condition} for \(i \in \{1, \dots, l_w\}\), where the decision variables are the control inputs \({u}_{i}^\mathcal{W}\).
\end{remark}


We tackle the \textbf{second challenge} by choosing a distance measure in \eqref{subeq: distance to the target set} which is linear on the decision variable $ \bm{d}_{\mathrm{exp}}$ and thus changing a non-convex problem \eqref{eq: exploration optimization ideal} into a convex problem. The distance measure we choose is defined as follows:
\begin{equation} \label{eq: distance measure definition}
\begin{aligned}
    & \mathrm{dist}({x}_{k,\mathrm{exp}}^j, \mathcal{SS}^{j-1}) = \min_{x_c} ||\Bar{x}_{k,\mathrm{exp}}^j-x_c||_2 \\
    & \qquad \qquad \qquad \qquad \quad ~~ \textnormal{s.t.,} ~ {\bm{\nu}^{j\star}_{k}}^\top H x_c = {\bm{\nu}^{j\star}_{k}}^\top h,  \\
    & \qquad \qquad \qquad \qquad \qquad ~~~~ \Bar{x}_{k,\mathrm{exp}}^j = A x_k^j + B ({u}^{\star} + \bm{d}_{\mathrm{exp}}),
\end{aligned}
\end{equation}
where $\bm{\nu}^{j\star}_{k}$ is a dual variable of \eqref{subeq: terminal constraints reform MPC} at time step $k$.
By solving the problem \eqref{eq: distance measure definition}, $\mathrm{dist}({x}_{k,\mathrm{exp}}^j, \mathcal{SS}^{j-1}) = {\bm{\nu}^{j\star}_{k}}^\top H B \bm{d}_{\mathrm{exp}}$. 
\begin{remark}
    If a non-zero dual variable $\bm{\nu}^{j\star}_{k} \neq \bm{0}$ exists in \eqref{subeq: terminal constraints reform MPC}, at least one constraint in \eqref{subeq: terminal constraints reform MPC} is active, with non-zero elements of $\bm{\nu}^{j\star}_{k}$ corresponding to these active constraints, due to dual feasibility and complementary slackness of the KKT conditions \cite{boyd2004convex}. The active constraints are then given by ${\bm{\nu}^{j\star}_{k}}^\top H x_c = {\bm{\nu}^{j\star}_{k}}^\top h$, as in \eqref{eq: distance measure definition}. Thus, the problem in \eqref{subeq: terminal constraints reform MPC} is to compute the distance from the state $\Bar{x}_{k,\mathrm{exp}}^j$ to the active constraint, as shown in Fig. \ref{fig:exploration cost clarification}.
    \begin{figure}[H]
    \begin{center}
    \includegraphics[width=0.6\linewidth,keepaspectratio]{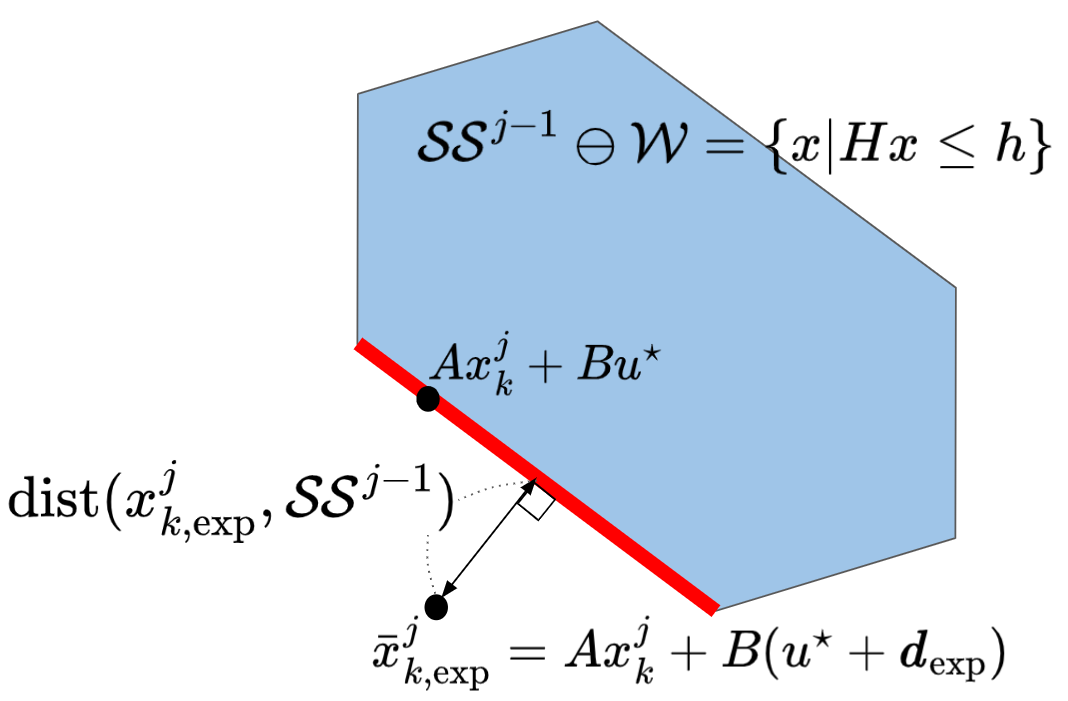}
    \caption{A two-dimensional example of the distance measure \eqref{eq: distance measure definition}. The light blue polytope is the nominal safe set $\mathcal{SS}^{j-1} \ominus \mathcal{W} = \{x | H x \leq h\}$ in \eqref{subeq: terminal constraints reform MPC}. The red line represents the active constraint.}
    \label{fig:exploration cost clarification}
    \end{center}
    \end{figure}
\end{remark}

In summary, if there exists a non-zero dual variable of \eqref{subeq: terminal constraints reform MPC}, i.e., $\bm{\nu}^{j\star}_{k} \neq \bm{0}$, the exploration problem \eqref{eq: exploration optimization ideal} is reformulated by replacing the control policy $\pi_{\mathrm{exp}}(\cdot)$ with $\hat{\pi}_{\mathrm{exp}}(\cdot)$ in \eqref{eq: convex control policy} as described in Remark \ref{rmk : replacing the control policy}, explicitly calculating the distance measure as in \eqref{eq: distance measure definition}, and robustifying the constraints. The tractable reformulation of \eqref{eq: exploration optimization ideal} for $x_{k}^j$ is given as follows:
\begin{equation} \label{eq: exploration optimization}
    \begin{aligned}
    & \hat{d}_{\mathcal{SS}^j}(x_k^j,{u}^{\star}) = \\
    & \max_{\substack{\bm{d}_{\mathrm{exp}}, \\ {u}_{1:l_w}^\mathcal{W}}} ~~ {\bm{\nu}^{j\star}_{k}}^\top H B \bm{d}_{\mathrm{exp}} \\
    &  ~~ \textnormal{s.t.,} ~~ \nabla f_k^j({u}^{\star})^\top {\bm{d}_{\mathrm{exp}}} \leq 0  \\
    & \quad \quad  \, ~~ \Bar{x}_{k,\mathrm{exp}}^j = A x_k^j + B ({u}^{\star} + \bm{d}_{\mathrm{exp}}), \\
    & \quad \quad  \, ~~\Bar{x}_{k,\mathrm{exp}}^j \in \mathcal{X}\ominus \mathcal{W}, ~{u}^{\star} + \bm{d}_{\mathrm{exp}} \in \mathcal{U}, \\
    & \quad \quad  \, ~~A {x}_{i}^\mathcal{W} + B {u}_{i}^\mathcal{W} \in \mathcal{SS}^{j-1} \ominus \mathcal{W}, \\
    & \quad \quad  \, ~~{x}_{i}^\mathcal{W} = A x_k^j + B ({u}^{\star} + \bm{d}_{\mathrm{exp}}) + v_i^\mathcal{W},\\
    & \quad \quad  \, ~~ {u}_{i}^\mathcal{W} \in \mathcal{U}, \\
    & \quad \quad  \, ~~ i\in \{1,\cdots,l_w\}.
    \end{aligned}
\end{equation}

When none of the constraints in \eqref{subeq: terminal constraints reform MPC} are active, the dual variable \( \bm{\nu}^{j\star}_{k} \) becomes zero. Then, the objective function of \eqref{eq: exploration optimization} evaluates to zero, offering no information about the distance to \( \mathcal{SS}^j \). In such cases, we explore within \( \mathcal{SS}^j \) by randomly sampling states around $x_k^j$ rather than solving \eqref{eq: exploration optimization} as follows:
\begin{equation} \label{eq: sampling in SS}
    x_i^{\mathcal{W}\star} \sim p_\text{sample}(x), ~i\in \{1,\cdots,l_w\}.
\end{equation}
where $p_\text{sample}(x)$ is a truncated multi-dimensional Gaussian distribution with mean vector \( x_k^j \) and covariance matrix equal to the identity matrix, truncated to the support \( \mathcal{SS}^j \). 
Later, we will show in Proposition \ref{prop: SS0 is rpi} and Corollary \ref{prop: piQ is feasible} that \eqref{eq:MPC discrete} is feasible for all \( x \in \mathcal{SS}^j \). This implies that  \eqref{eq:MPC discrete} is feasible at \( x_i^{\mathcal{W}\star} \) \eqref{eq: sampling in SS}.

After solving \eqref{eq: exploration optimization} or sampling using \eqref{eq: sampling in SS} for $x_{k}^j$ for all $k \in \{0,\cdots,T^j\}$, we determine states to be explored as follows:
\begin{equation} \label{eq: exp state matrix}
    \mathbf{X}_{k,\mathrm{exp}}^j = [{x}_{1}^{\mathcal{W}\star}, \cdots, ~{x}_{l_w}^{\mathcal{W}\star}].
\end{equation}
The states \eqref{eq: exp state matrix} are utilized to update the safe set $\mathcal{SS}^{j-1}$, which will be introduced in Sec. \ref{sec: value function update}.

To calculate the corresponding input to each explored state in $\mathbf{X}_{k,\mathrm{exp}}^j$, we solve \eqref{eq:MPC discrete} for each element in $\mathbf{X}_{k,\mathrm{exp}}^j$ as the initial condition and collect the optimal inputs as follows:
\begin{equation} \label{eq: exp input matrix}
\begin{aligned}
    & \mathbf{U}_{k,\mathrm{exp}}^j = [\pi^{j\star}({x}_{1}^{\mathcal{W}\star}), \cdots, \pi^{j\star}({x}_{l_w}^{\mathcal{W}\star})],
\end{aligned}
\end{equation}
for all $k \in \{0,\cdots,T^j\}$.
Remind that $\pi^{j\star}(\cdot)$ denotes the MPC policy \eqref{eq:MPC policy}. 
All elements in $\mathbf{X}_{k,\mathrm{exp}}^j$ are feasible points of the problem \eqref{eq:MPC discrete} because the optimal solution ${u}_{i}^\mathcal{W}$, where $i \in \{1,\cdots,l_w\}$, is a feasible input of the problem \eqref{eq:MPC discrete} for all elements in $\mathbf{X}_{k,\mathrm{exp}}^j$.

When solving \eqref{eq:MPC discrete} to construct $\mathbf{U}_{k,\mathrm{exp}}^j$ \eqref{eq: exp input matrix}, we collect the corresponding optimal cost of \eqref{eq:MPC discrete} as follows:
\begin{equation} \label{eq: exp cost matrix}
\begin{aligned}
    & \mathbf{J}_{k,\mathrm{exp}}^j = [\hat{V}^{j}({x}_{1}^{\mathcal{W}\star}), \cdots, \hat{V}^{j}({x}_{l_w}^{\mathcal{W}\star})],
\end{aligned}
\end{equation}
for all $k \in \{0,\cdots,T^j\}$.
\subsection{Update the safe set $\mathcal{SS}^j$ and the value function $Q^j(\cdot)$} \label{sec: value function update}
After the $j$-th task and the exploration are completed, $\mathcal{SS}^{j-1}$ is updated.
We augment the state and input data as follows:
\begin{equation} \label{eq: state data}
\begin{split}
    & \mathbf{X}^{j} = [\mathbf{X}^{j-1}, \underbrace{x_0^j, x_1^j,\cdots, x_{T^j}^j}_\text{Online}, \underbrace{\mathbf{X}_{0,\mathrm{exp}}^j,\cdots,  ~\mathbf{X}_{T^j,\mathrm{exp}}^j}_\text{From exploration}],
\end{split}
\end{equation}
\begin{equation} \label{eq: input data}
\begin{split}
    & \mathbf{U}^{j} = [\mathbf{U}^{j-1}, \pi^{j\star}(x_0^j),\cdots, \pi^{j\star}(x_{T^j}^j),  \mathbf{U}_{0,\mathrm{exp}}^j,\cdots,  ~\mathbf{U}_{T^j,\mathrm{exp}}^j].
\end{split}
\end{equation}
Moreover, the cost data is updated by augmenting the optimal costs of \eqref{eq:MPC discrete} corresponds to the state and input data as follows:
\begin{equation} \label{eq: cost data}
    \begin{aligned}
        & \mathbf{J}^{j} = [\mathbf{J}^{j-1}, \hat{V}^{j}(x_0^j),\cdots, \hat{V}^{j}(x_{T^j}^j),  \mathbf{J}_{0,\mathrm{exp}}^j,\cdots, \mathbf{J}_{T^j,\mathrm{exp}}^j].
    \end{aligned}
\end{equation}
We update $\mathcal{SS}^{j}$ by taking the convex hull operation of all elements in $\mathbf{X}^{j}$ as follows:
\begin{equation} \label{eq: terminal constraint update}
\begin{split}
    & \mathcal{SS}^{j} = \mathrm{conv}\Biggl\{\bigcup_{i = 1}^{|\mathbf{X}^{j}|} [\mathbf{X}^{j}]_i \Biggr\},
\end{split}
\end{equation}
where $|\mathbf{X}^{j}|$ is the number of states stored in the state matrix $\mathbf{X}^{j}$, i.e., $\mathbf{X}^{j} \in \mathbb{R}^{n_x \times |\mathbf{X}^{j}|}$. Note that as we augment the state data matrix as shown in \eqref{eq: state data}, the convex hull operation of them will continue to expand, i.e., $\mathcal{SS}^{j-1} \subset \mathcal{SS}^{j}$.

We update $Q^{j}(\cdot)$ by solving the following problem:
\begin{equation} \label{eq: value function update}
\begin{split}
    & Q^{j}(x) = \min_{\bm{\lambda}^{j} \in \mathbb{R}^{|\mathbf{X}^{j}|}}~ \mathbf{J}^{j} \bm{\lambda}^{j} \\
    & \qquad \qquad ~~ \textnormal{s.t.,} ~~~~ \mathbf{X}^{j} \bm{\lambda}^{j} = x, \\
    & \qquad \qquad \qquad ~~~~ \, \bm{\lambda}^{j} \geq \bm{0}, ~ \bm{1}^\top \bm{\lambda}^{j} = 1.
\end{split}
\end{equation}
\( Q^j(x) \) represents the minimum cost for state \( x \) in the \( j \)-th scenario. The vector \( \bm{\lambda}^j \) serves as convex coefficients, ensuring a valid convex combination of outcomes. The optimization minimizes the cost \( \mathbf{J}^j \bm{\lambda}^j \) while ensuring that the outcomes \( \mathbf{X}^j \bm{\lambda}^j \) correspond to the given state \( x \). The minimum convex coefficients are chosen to improve performance over iterations, as shown in Proposition 4.
$Q^{j}(x)$ is a convex, piecewise affine function as it is an optimal objective function of the multi-parametric linear program \cite{borrelli2017predictive}. 
When calculating the optimal input in \eqref{eq:MPC discrete} at episode $j$, the function $Q^j(x)$ can be incorporated into the problem \eqref{eq:MPC discrete} and form a single unified optimization problem. 

The function $Q^{j}(x)$ is defined on $\mathcal{SS}^{j}$ because of $\mathbf{X}^{j} \bm{\lambda}^{j} = x$ in \eqref{eq: value function update}. This constraint restricts its domain to the states expressable as a convex combination of the recorded states in \eqref{eq: state data}, which is equal to $\mathcal{SS}^{j}$ by construction \eqref{eq: terminal constraint update}. 

Since $Q^{j}(x)$ is defined on the bounded domain $\mathcal{SS}^{j}$ and is piecewise affine, continuous, and bounded above, it is Lipschitz continuous in $\mathcal{SS}^{j}$ defined as follows:
\begin{equation} \label{eq: lip con Qj}
    \lVert Q^j(x) - Q^j(y) \rVert \leq L^j \lVert x - y \rVert, ~\forall x, \ y \in \mathcal{SS}^j,
\end{equation}
where $L^j > 0$ is a Lipschitz constant.

Let $\bm{\lambda}^{j\star}$ denote the optimal solution of \eqref{eq: initial value function} when $j=0$ and of \eqref{eq: value function update} when $j \geq 1$.
We define the control policy $\pi^{j}_\mathrm{Q}(\cdot)$ is defined for all $x \in \mathcal{SS}^{j}$ as follows:
\begin{equation} \label{eq: safe policy}
\begin{split}
    & \pi^{j}_\mathrm{Q}(x) = \mathbf{U}^{j} \bm{\lambda}^{j\star},
\end{split}
\end{equation}
which will be used to prove the robust positive invariance of the safe set $\mathcal{SS}^j$ and the pointwise non-increasing property of the function $Q^j(\cdot)$ for all $x \in \mathcal{SS}^j$.

\subsection{Properties of $\mathcal{SS}^j$ and $Q^j(\cdot)$} \label{subsec: properties of SS and Q}
We introduce properties of the safe set $\mathcal{SS}^j$ \eqref{eq: terminal constraint update} and the value function $Q^j(\cdot)$ \eqref{eq: value function update}. 
\begin{proposition} \label{prop: SS0 is rpi}
(Robust Positive Invariance of $\mathcal{SS}^{j}$) Let Assumptions \ref{assum: bounded noise}-\ref{assum: initial solution} hold.
Then, $\forall j \geq 0,$ $\mathcal{SS}^{j}$ is a robust positive invariant set for the closed-loop system \eqref{eq:system} controlled by a policy $\pi^{j}_\mathrm{Q}(\cdot)$ \eqref{eq: safe policy}.
\end{proposition}
\begin{proof}
See the Appendix \ref{sec: proof of rpi}.
\end{proof}
\begin{corollary}\label{prop: piQ is feasible}
Let Assumptions \ref{assum: bounded noise}-\ref{assum: initial solution} hold.
\( \pi^{j}_\mathrm{Q}(x) \) is a feasible solution to the MPC problem \eqref{eq:MPC discrete} for all $x \in \mathcal{SS}^{j}$.
\end{corollary} 
\begin{proof}
The closed-loop system \eqref{eq:system}, controlled by \( \pi^{j}_\mathrm{Q}(x) \), remains within \( \mathcal{SS}^{j} \) for \( k \geq 0 \) since \( \mathcal{SS}^{j} \) is a robust positive invariant set for this system. Also, \( \mathcal{SS}^{j} \subset \mathcal{X} \), so the closed-loop system satisfies the state constraints in \eqref{eq:constraints}. Moreover, since \( \pi^{j}_\mathrm{Q}(\cdot) \) is a convex combination of inputs in \( \mathcal{U} \), and \( \mathcal{U} \) is a convex set as per \eqref{eq:constraints}, it follows that \( \pi^{j}_\mathrm{Q}(x) \in \mathcal{U} \). Thus, as all constraints in \eqref{eq:MPC discrete} are satisfied, the claim is proved.
\end{proof}

\begin{proposition} \label{prop: nonincreasing value function}
(Pointwise non-increasing of $Q^j(\cdot)$) Let Assumptions \ref{assum: bounded noise}-\ref{assum: initial solution} hold. 
Consider the value function $Q^j(\cdot)$ in \eqref{eq: value function update}.
$\forall j \geq 0$, the value function does not increase pointwise:
\begin{equation} \label{eq: decreasing}
\forall x\in\mathcal{SS}^{j}, ~ Q^{j}(x) \geq Q^{j+1}(x).
\end{equation}
\end{proposition}
\begin{proof}
    For any $x \in \mathcal{SS}^j$, suppose $\bm{\lambda}^{j\star}$ is an optimal solution of \eqref{eq: value function update} at episode $j$. Then, from \eqref{eq: value function update}, we have that:
    \begin{equation*}
    \begin{aligned}
        & Q^j(x) = \mathbf{J}^j \bm{\lambda}^{j\star} 
    \end{aligned}
    \end{equation*}

    Note that we augment the state and cost data matrices, $\mathbf{X}^j$ and $\mathbf{J}^j$, to define $\mathbf{X}^{j+1}$ and $\mathbf{J}^{j+1}$ as in \eqref{eq: state data} and \eqref{eq: cost data}, respectively. Consequently, we define $\bm{\lambda}^{j+1}_\mathrm{cand} = [\bm{\lambda}^{j\star \top}, 0, 0, \dots, 0]^\top \in \mathbb{R}^{\mathrm{dim}(\mathbf{J}^{j+1})}$, ensuring that $x = \mathbf{X}^j \bm{\lambda}^{j\star} = \mathbf{X}^{j+1} \bm{\lambda}^{j+1}_\mathrm{cand}$ and $\mathbf{J}^j \bm{\lambda}^{j\star} = \mathbf{J}^{j+1}\bm{\lambda}^{j+1}_\mathrm{cand}$.
    Moreover, $\bm{\lambda}^{j+1}_\mathrm{cand} \geq \bm{0}$ and $\bm{1}^\top \bm{\lambda}^{j+1}_\mathrm{cand} = 1$, implying that $\bm{\lambda}^{j+1}_\mathrm{cand}$ is a feasible solution to the problem in \eqref{eq: value function update} at episode $j+1$.
    As $Q^{j+1}(x)$ is the optimal cost function of \eqref{eq: value function update} at episode $j+1$. we have the following:
    \begin{equation} \label{eq: Qj > Qj1}
    \begin{aligned}
        & Q^j(x) = \mathbf{J}^j \bm{\lambda}^{j\star}  =  \mathbf{J}^{j+1}\bm{\lambda}^{j+1}_\mathrm{cand} \geq Q^{j+1}(x).
    \end{aligned}
    \end{equation}
    This proves the claim.
\end{proof}

\begin{proposition} \label{prop: (17) hold for all j}
(\eqref{eq: Q0 condition} can be generalized.) Let Assumptions \ref{assum: bounded noise}-\ref{assum: initial solution} hold. 
Consider $\mathbf{X}^{j}$ and $\mathbf{U}^{j}$ in \eqref{eq: state data} and \eqref{eq: input data}, $\mathbf{J}^{j}$ in \eqref{eq: cost data}, and $Q^j(\cdot)$ in \eqref{eq: value function update}.
Then, the following holds $\forall j \geq 0$:
\begin{equation} \label{eq: Qj condition}
    \begin{aligned}
        & [\mathbf{J}^{j}]_i \geq \ell([\mathbf{X}^{j}]_i,[\mathbf{U}^{j}]_i) \\
        & ~~~~~~~~~~~~~~~ + \sum_{m=1}^M p_m Q^j(A[\mathbf{X}^{j}]_i + B[\mathbf{U}^{j}]_i + \tilde{w}_m).
    \end{aligned}
\end{equation}
\end{proposition}
\begin{proof}
    We will prove the claim by induction.
    For the base case, when $j=0$, the claim holds by \eqref{eq: Q0 condition}. 
    Now, use an induction step. Suppose for some $j-1 \geq 0$, the claim holds as follows:
    \begin{equation*}
        \begin{aligned}
            & [\mathbf{J}^{j-1}]_i \geq \ell([\mathbf{X}^{j-1}]_i,[\mathbf{U}^{j-1}]_i) \\
            & ~~~~~~~~~~~~ + {\scriptstyle\sum}_{m=1}^M p_m Q^{j-1}(A[\mathbf{X}^{j-1}]_i + B[\mathbf{U}^{j-1}]_i + \tilde{w}_m).
        \end{aligned}
    \end{equation*}
    
    The data matrices $\mathbf{X}^{j}$, $\mathbf{U}^{j}$, and $\mathbf{J}^{j}$ are constructed by adding new data to $\mathbf{X}^{j-1}$, $\mathbf{U}^{j-1}$, and $\mathbf{J}^{j-1}$ as described in \eqref{eq: state data}, \eqref{eq: input data}, and \eqref{eq: cost data}, respectively. Thus, for the existing parts, i.e., $i \leq |\mathbf{J}^{j-1}|$, $[\mathbf{X}^{j}]_i = [\mathbf{X}^{j-1}]_i$, $[\mathbf{U}^{j}]_i = [\mathbf{U}^{j-1}]_i$, and $[\mathbf{J}^{j}]_i = [\mathbf{J}^{j-1}]_i$. Moreover, from Proposition \ref{prop: nonincreasing value function}, $Q^{j-1}(x) \geq Q^j(x), ~\forall x \in \mathcal{SS}^{j-1}$. Thus, from the induction hypothesis, we have the following for $i \leq |\mathbf{J}^{j-1}|$:
    \begin{equation} \label{eq: existing data}
        \begin{aligned}
            & [\mathbf{J}^{j}]_i \geq \ell([\mathbf{X}^{j}]_i,[\mathbf{U}^{j}]_i) \\
            & ~~~~~~~~~~~~ + {\scriptstyle\sum}_{m=1}^M p_m Q^{j}(A[\mathbf{X}^{j}]_i + B[\mathbf{U}^{j}]_i + \tilde{w}_m).
        \end{aligned}
    \end{equation} 
    
    For the newly added parts, i.e., $i > |\mathbf{J}^{j-1}|$, the following holds because they are the optimal cost of \eqref{eq:MPC discrete} and the associated $[\mathbf{U}^{j}]_i$ is the optimal input of \eqref{eq:MPC discrete} at the state $[\mathbf{X}^{j}]_i$:
    \begin{equation} \label{eq: newly added cost data}
        \begin{aligned}
            & [\mathbf{J}^{j}]_i = \ell([\mathbf{X}^{j}]_i,[\mathbf{U}^{j}]_i) \\
            & ~~~~~~~~~~~~ + {\scriptstyle\sum}_{m=1}^M p_m Q^{j}(A[\mathbf{X}^{j}]_i + B[\mathbf{U}^{j}]_i + \tilde{w}_m).
        \end{aligned}
    \end{equation}

    From \eqref{eq: existing data} and \eqref{eq: newly added cost data}, the claim is proved by induction.
\end{proof}

\section{Controller Properties}
In this section, we show the main properties of the proposed control algorithm.

\subsection{Feasibility} \label{subsec: feasibility}
First, we prove the recursive feasibility of the MPC \eqref{eq:MPC discrete}.
\begin{theorem} \label{thm: recursive feasibility}
Let Assumptions \ref{assum: bounded noise}-\ref{assum: initial solution} hold.
Then, for all time steps $k \geq 0$ and all episodes $j \geq 1$, the MPC \eqref{eq:MPC discrete} is feasible if the state $x_k^j$ is obtained by applying the MPC policy \eqref{eq:MPC policy} to the system \eqref{eq:system} starting from the initial state $x_S$. 
\end{theorem}
\begin{proof}
We will prove the claim by induction.
For the base case, we prove that the problem \eqref{eq:MPC discrete} is feasible at $k=0, ~\forall j \geq 1$.
From Assumption \ref{assum: initial solution}, a feasible solution exists starting from $x_0^j = x_S$. 

Now use an induction step.
Suppose the problem \eqref{eq:MPC discrete} is feasible at some $k\geq 0$ at episode $j \geq 1$.
From the induction hypothesis, $Ax_k^j + B\pi^{j \star}(x_k^j) \in \mathcal{SS}^{j-1} \ominus \mathcal{W}$, which is the terminal constraint of \eqref{eq:MPC discrete} and implies $x_{k+1}^j \in \mathcal{SS}^{j-1}, ~\forall w_k^j \in \mathcal{W}$.
From Corollary \ref{prop: piQ is feasible}, $\pi^{j-1}_\mathrm{Q}(x_{k+1}^j)$ is a feasible solution to \eqref{eq:MPC discrete} at $k+1$ at $j$.
By induction, the claim is proved.
\end{proof}

Based on the Theorem \ref{thm: recursive feasibility}, we can prove the feasibility of the exploration problem \eqref{eq: exploration optimization}.
\begin{corollary} \label{cor: exp feasible}
Consider the exploration problem \eqref{eq: exploration optimization} where the argument is the state $x_k^j$, the primal optimal solution of the MPC problem in \eqref{eq:MPC discrete} is ${u}_k^{j \star}$, and the corresponding dual optimal solutions are $\bm{\nu}^{j\star}_{k}$ and $\bm{\gamma}^{j\star}_{k}$.
Let Assumptions \ref{assum: bounded noise}-\ref{assum: initial solution} hold.
Then the exploration problem \eqref{eq: exploration optimization} has a feasible solution.
\end{corollary}
\begin{proof}
    Consider the case $\bm{d}_{k,\mathrm{exp}}^j = \bm{0}$. Then, in \eqref{eq: exploration optimization}, ${x}_{i}^\mathcal{W} = A x_k^j + B {u}^{\star} + v_i^\mathcal{W}$. Since $v_i^\mathcal{W}$ is one of the vertices of $\mathcal{W}$ and ${u}^{\star}$ is an optimal input of \eqref{eq:MPC reformulated}, ${x}_{i}^\mathcal{W} \in \mathcal{SS}^{j-1}$.
    From Proposition \ref{prop: SS0 is rpi}, $\mathcal{SS}^{j-1}$ is a robust positive invariant set for the closed-loop system \eqref{eq:system} controlled by $\pi^{j-1}_\mathrm{Q}(\cdot)$ \eqref{eq: safe policy}. Thus, $A{x}_{i}^\mathcal{W} + B\pi^{j-1}_\mathrm{Q}({x}_{i}^\mathcal{W}) + w \in \mathcal{SS}^{j-1}, ~\forall w \in \mathcal{W}$.
    
    Therefore, we can find a feasible point of \eqref{eq: exploration optimization}:
    \begin{equation}
    \begin{aligned}
        & \bm{d}_{k,\mathrm{exp}}^j = \bm{0},\\
        & {u}_{i}^\mathcal{W} = \pi^{j-1}_\mathrm{Q}(Ax_k^j + B{u}^{\star}+v_i^\mathcal{W}). 
    \end{aligned}
    \end{equation} 
    Thus, the problem \eqref{eq: exploration optimization} has a feasible solution.
\end{proof}

\subsection{Convergence} \label{subsec: Convergence}
In this subsection, we prove the convergence in probability, $\lim_{k \rightarrow \infty} \mathbb{P}(x_k^j \in \mathcal{O}) = 1$.
We first prove the following.
\begin{proposition} \label{prop: Qj > ell + sum pmQj}
For $j\geq 0$, consider the value function $Q^{j}(\cdot)$ in \eqref{eq: value function update} and the LMPC policy \eqref{eq:MPC policy} at episode $j+1$, i.e., $\pi^{j+1\star}(\cdot)$.
Let Assumptions \ref{assum: bounded noise}-\ref{assum: initial solution} hold.
Then, the value function $Q^{j}(\cdot)$ satisfies the following inequality for all $x \in \mathcal{SS}^{j}$:
\begin{equation*}
    Q^{j}(x) \geq \ell(x, \pi^{j+1\star}(x)) + \sum_{m=1}^{M} p_m Q^{j}(Ax + B\pi^{j+1\star}(x) + \tilde{w}_m).
\end{equation*}
\end{proposition}
\begin{proof}
    Let $\bm{\lambda}^{j\star}$ denote an optimal solution of \eqref{eq: value function update} in the episode $j$, i.e., $Q^j(x) = \mathbf{J}^j\bm{\lambda}^{j\star}$ for given $x\in\mathcal{SS}^{j}$.
    Multiplying $\bm{\lambda}^{j\star}$ on both sides of the \eqref{eq: Qj condition} in Proposition \ref{prop: (17) hold for all j}, we have the following for $x\in\mathcal{SS}^{j}$ as follows:
    \begin{equation*} 
    \begin{aligned}
        & Q^j(x) \geq {\scriptstyle\sum}_{i=1}^{\mathrm{dim}(\mathbf{J}^j)} [\bm{\lambda}^{j\star}]_i \big\{\ell([\mathbf{X}^{j}]_i,[\mathbf{U}^{j}]_i) \\
        & ~~~~~~~~~~~~~~~~~~~~ + {\scriptstyle\sum}_{m=1}^M p_m Q^{j}(A[\mathbf{X}^{j}]_i + B[\mathbf{U}^{j}]_i + \tilde{w}_m) \big\}
    \end{aligned}
    \end{equation*}
    Due to the convexity of the stage cost $\ell(\cdot,\cdot)$ from Assumption \ref{assum: Stage cost} and the convexity of the value function $Q^j(\cdot)$, as it arises from being the optimal objective function of the mp-LP problem \eqref{eq: value function update}, the following holds for $x \in \mathcal{SS}^{j}$:
    \begin{equation*} 
    \begin{aligned}
        & Q^j(x) \\
        & \geq \ell\bigg(\underbrace{{\scriptstyle\sum}_{i=1}^{\mathrm{dim}(\mathbf{J}^j)} [\bm{\lambda}^{j\star}]_i [\mathbf{X}^{j}]_i}_{x \text{ from \eqref{eq: value function update}}},~\underbrace{{\scriptstyle\sum}_{i=1}^{\mathrm{dim}(\mathbf{J}^j)} [\bm{\lambda}^{j\star}]_i [\mathbf{U}^{j}]_i}_{\pi_\mathrm{Q}^j(x) \text{ from \eqref{eq: safe policy}}}\bigg) \\
        & ~~ + {\scriptstyle\sum}_{m=1}^M p_m Q^{j}\bigg( {\scriptstyle\sum}_{i=1}^{\mathrm{dim}(\mathbf{J}^j)} [\bm{\lambda}^{j\star}]_i (A[\mathbf{X}^{j}]_i + B[\mathbf{U}^{j}]_i + \tilde{w}_m)\bigg) \\
        & = \ell(x,\pi_\mathrm{Q}^j(x)) + {\scriptstyle\sum}_{m=1}^M p_m Q^{j}(Ax + B\pi_\mathrm{Q}^j(x) + \tilde{w}_m). \\
    \end{aligned}
    \end{equation*}
    The inequality holds as \( \bm{\lambda}^{j\star} \) is an optimal solution of \eqref{eq: value function update}, giving \(\sum_{i=1}^{\mathrm{dim}(\mathbf{J}^j)} [\bm{\lambda}^{j\star}]_i [\mathbf{X}^{j}]_i = x\) from the first constraint of \eqref{eq: value function update}. Also, \(\sum_{i=1}^{\mathrm{dim}(\mathbf{J}^j)} [\bm{\lambda}^{j\star}]_i [\mathbf{U}^{j}]_i = \pi_\mathrm{Q}^j(x)\) from \eqref{eq: safe policy}.

    From Corollary \ref{prop: piQ is feasible}, any $\pi_\mathrm{Q}^j(x)$ is a feasible solution of the MPC \eqref{eq:MPC discrete} for all $x \in \mathcal{SS}^j, ~\forall k \geq 0$. 
    In the MPC \eqref{eq:MPC discrete} at episode $j$, the cost $\ell(x,u) + \sum_{m=1}^M p_m Q^{j}(Ax + Bu + \tilde{w}_m)$ is minimized so we have that:
    \begin{equation*}
    \begin{aligned}
        & Q^j(x) \\
        &\geq \ell(x,\pi_\mathrm{Q}^j(x)) + {\scriptstyle\sum}_{m=1}^M p_m Q^{j}(Ax + B\pi_\mathrm{Q}^j(x) + \tilde{w}_m) \\
        &\geq \ell(x,\pi^{j+1\star}(x)) + {\scriptstyle\sum}_{m=1}^M p_m Q^{j}(Ax + B\pi^{j+1\star}(x) + \tilde{w}_m),
    \end{aligned}
    \end{equation*}
    which proves the claim.
\end{proof}

\begin{theorem}
    For $j\geq 1$, consider the closed-loop system \eqref{eq:system} controlled by the MPC policy \eqref{eq:MPC policy} at episode $j$. Let Assumptions \ref{assum: bounded noise}-\ref{assum: initial solution} hold.
    Then, $\lim_{k \rightarrow \infty} \mathbb{P}(x_k^j \in \mathcal{O}) = 1$.
\end{theorem}
\begin{proof}
From the Proposition \ref{prop: lyapunov function of the closed-loop system} and \ref{prop: Qj > ell + sum pmQj}, the following inequality holds:
\begin{equation*}
    Q^{j-1}(x) \geq \ell(x, \pi^{j \star}(x)) + \mathbb{E}_w[Q^{j-1}(Ax + B\pi^{j \star}(x) + w)],
\end{equation*}
for all $x \in \mathcal{SS}^{j}$.
Then, starting from $x_0^j = x_S$, we can derive the following inequalities:
\begin{equation} \label{eq: decreasing sequence sum of stage costs1}
\begin{aligned}
    & Q^{j-1}(x_S) \\
    & \geq \mathbb{E}_{w_0^j}[\ell(x_0^j, \pi^{j \star}(x_0^j)) + Q^{j-1}(x_1^j)] \\
    & \geq \mathbb{E}_{w_{0:1}^j}[\ell(x_0^j, \pi^{j \star}(x_0^j)) + \ell(x_1^j, \pi^{j \star}(x_1^j)) + Q^{j-1}(x_2^j)] \\
    & \cdots \\
    & \geq \mathbb{E}_{w_{0:\infty}^j}\Big[{\scriptstyle\sum}_{k=0}^\infty \ell(x_k^j, ~  \pi^{j \star}(x_k^j)) + Q^{j-1}(x_\infty^j) \Big]. \\
\end{aligned}
\end{equation}
By definition of the value function, $Q^{j-1}(\cdot) \geq 0$.
Thus, the following inequalities hold from \eqref{eq: decreasing sequence sum of stage costs1}:
\begin{equation} \label{eq: decreasing sequence sum of stage costs}
\begin{aligned}
    & Q^{j-1}(x_S) \geq \mathbb{E}_{w_{0:\infty}^j}\Big[{\scriptstyle\sum}_{k=0}^\infty \ell(x_k^j, ~ \pi^{j \star}(x_k^j)) \Big]. 
\end{aligned}
\end{equation}
We know that from Proposition \ref{prop: nonincreasing value function}, $Q^{j-1}(x_S) \leq Q^{0}(x_S)$.
Moreover, from Assumption \ref{assum: initial solution} and \eqref{eq: initial value function}, $Q^{0}(x_S) < \infty$. 

Let $S_n = \mathbb{E}_{w_{0:n-1}^j}\Bigr[\sum_{k=0}^n \ell(x_k^j, \pi^{j \star}(x_k^j)) \Bigr]$ for $n \geq 1$.
Since $\ell(\cdot, \cdot) \geq 0$ by Assumption \ref{assum: Stage cost}, $S_n$ is monotonically increasing with $n$.
Moreover, it is bounded above by \eqref{eq: decreasing sequence sum of stage costs}.
Therefore, we have that $\lim_{n \rightarrow \infty} S_n = \mathbb{E}_{w_{0:\infty}^j}\Bigr[\sum_{k=0}^\infty \ell(x_k^j, \pi^{j \star}(x_k^j)) \Bigr]$ converges to a finite value \cite{rudin1976principles}.
Thus, we can conclude that $\lim_{k \rightarrow \infty} \mathbb{E}_{w_k^j}[\ell(x_k^j, \pi^{j \star}(x_k^j))] = 0$ \cite{rudin1976principles}. By Assumption \ref{assum: Stage cost}, this implies $\lim_{k \rightarrow \infty} \mathbb{P}(x_k^j \in \mathcal{O}) = 1$, which proves the claim.
\end{proof}
\begin{remark}
    \eqref{eq: decreasing sequence sum of stage costs} shows that $Q^{j-1}(x_S)$ is a performance upper bound of the expected cost of the system \eqref{eq:system} controlled by the MPC policy \eqref{eq:MPC policy} at episode $j$. Moreover, by Proposition \ref{prop: nonincreasing value function}, this upper bound decreases over episodes, which is one of our contributions.
\end{remark}

\subsection{Optimality} \label{subsec: optimiality gap}
In this subsection, we show the value function at the initial state $x_S$, i.e., $Q^j(x_S)$ \eqref{eq: value function update}, converges close to the optimal value function, i.e., $Q^{\star}(x_S)$ \eqref{eq:ftocp}, as the episode $j \rightarrow \infty$ under some assumptions.
To begin with, we prove that the safe set $\mathcal{SS}^j$ and the value function $Q^j(\cdot)$ converges as $j \rightarrow \infty$.
\begin{proposition}
        As $j \rightarrow \infty$, the safe set and the value function converge as follows:
    \begin{equation*}
        \mathcal{SS}^j \rightarrow \mathcal{SS}^\infty, ~Q^j(\cdot) \rightarrow Q^\infty(\cdot).
    \end{equation*}
\end{proposition}
\begin{proof}
First, we prove that \(\mathcal{SS}^j \rightarrow \mathcal{SS}^\infty\). Note that as the state data matrix is augmented as in \eqref{eq: state data}, the convex hull operation on these states will continue to expand, i.e., \(\mathcal{SS}^{j-1} \subset \mathcal{SS}^j\). By \cite[Prop. 1.4]{resnick2013probability}, this expansion implies that there exists a limit \(\lim_{j \rightarrow \infty}\mathcal{SS}^j = \mathcal{SS}^\infty\).

Second, we prove that \(Q^j(\cdot) \rightarrow Q^\infty(\cdot)\). For each fixed \(x \in \mathcal{SS}^j\), consider the sequence \(\{Q^j(x)\}\) with respect to $j$. From Proposition \ref{prop: nonincreasing value function}, \(Q^{j-1}(x) \geq Q^j(x)\), indicating that \(\{Q^j(x)\}\) is a pointwise nonincreasing sequence of real numbers. Since \(Q^j(x)\) is bounded below for all \(j\) by Assumption \ref{assum: Stage cost} and \eqref{eq: decreasing sequence sum of stage costs}, \(\lim_{j\rightarrow \infty} Q^j(x)\) exists for each fixed \(x\) by the Monotone Convergence Theorem \cite[Thm. 3.14]{rudin1976principles}. This implies that the sequence of functions \(\{Q^j(\cdot)\}\) converges pointwise to a function \(Q^\infty(\cdot)\).
\end{proof}

\begin{definition}
    The control policy $\pi^\infty(\cdot)$ is the MPC policy \eqref{eq:MPC policy} when the safe set and the value function in the MPC problem \eqref{eq:MPC discrete} are set to $\mathcal{SS}^\infty$ and $Q^\infty(\cdot)$, respectively.
\end{definition}

We use the following assumptions.
\begin{assumption} \label{assum: it converges to steady-state solution}
   The closed-loop state at the time step $k$ of the system \eqref{eq:system} controlled by $\pi^\infty(\cdot)$, $x_k^\infty$, are located in the relative interior of the set $\mathcal{SS}^\infty$, and will reach the target set $\mathcal{O}$ within the finite time step $T^\infty$.
\end{assumption}
\begin{assumption} \label{assum: existence of Qstar}
    There exists an optimal value function of \eqref{eq:ftocp}, denoted as $Q^\star: \mathcal{X} \rightarrow \mathbb{R}$. Furthermore, the system \eqref{eq:system} controlled by the optimal control policy $\pi^\star(\cdot)$ will reach the target set $\mathcal{O}$ within the finite time step $T^\star$.
\end{assumption}

Note that the optimal value function $Q^\star(\cdot)$ is convex as it is a sum of convex stage costs as written in \eqref{eq:ftocp}. Once the system \eqref{eq:system} controlled by either $\pi^\infty(\cdot)$ or $\pi^\star(\cdot)$ reaches the target set $\mathcal{O}$, it will remain within it. This is because remaining in the target set $\mathcal{O}$ is optimal according to Assumption \ref{assum: Stage cost}. 
Moreover, remaining in the target set $\mathcal{O}$ is feasible by Assumption \ref{assum: Robust Positive Invariance}.


\begin{proposition} \label{prop: TQ tools}
    Let Assumptions \ref{assum: bounded noise}-\ref{assum: it converges to steady-state solution} hold. Let $\delta(w,M)$ satisfies the inequality in \eqref{eq: delta_bnd}, and $\delta = \max_{w \in \mathcal{W}} \delta(w,M)$. Then, the following inequalities hold for all $x \in \mathcal{SS}^\infty$:
    \begin{equation} \label{eq: prop TQ tools main}
        \mathbb{E}_w [Q^\infty(x+w)] - \sum_{m=1}^{M} p_m Q^\infty(x + \tilde{w}_m) \geq -L^\infty \delta,
    \end{equation}
    where $L^\infty$ is a Lipschitz constant of the function $Q^\infty(\cdot)$.
\end{proposition}
\begin{proof}
    For the $m \in \{1,\cdots,M\}$ such that $\mu_m(w) \neq 0$, we have the following:
    \begin{equation} \label{eq: Q - Qm below Ldelta}
        \begin{aligned}
            & \lVert Q^\infty(x + w) - Q^\infty(x + \tilde{w}_m) \rVert \leq L^\infty \lVert w - \tilde{w}_m \rVert \leq L^\infty \delta. 
        \end{aligned}
    \end{equation}
    The first inequality is from \eqref{eq: lip con Qj}, and the second inequality is due to \eqref{eq: delta_bnd} and $\delta(w,M) \leq \delta$ by definition.
    As $\sum_{m=1}^{M} \mu_m(w) = 1$, we can derive the following from \eqref{eq: Q - Qm below Ldelta}:
    \begin{equation} \label{eq: Q geq Qm - Ldelta}
        \begin{aligned}
            &  Q^\infty(x + w) \geq {\scriptstyle\sum}_{m=1}^{M} \mu_m(w) Q^\infty(x + \tilde{w}_m) - L^\infty \delta.
        \end{aligned}
    \end{equation}
    Taking expectations to both sides of \eqref{eq: Q geq Qm - Ldelta}, we have that:
    \begin{equation} \label{eq: prop EQ >= Qhat - Ldelta}
        \begin{aligned}
            & ~~~~ \mathbb{E}_w [Q^\infty(x+w)]  \\
            & \geq \mathbb{E}_w \big[{\scriptstyle\sum}_{m=1}^{M} \mu_m(w)Q^\infty(x + \tilde{w}_m)- L^\infty \delta \big]  \\
            & = {\scriptstyle\sum}_{m=1}^{M} \mathbb{E}_w [\mu_m(w)] Q^\infty(x + \tilde{w}_m) - L^\infty \delta \\
            & = {\scriptstyle\sum}_{m=1}^{M} p_m Q^\infty(x + \tilde{w}_m) - L^\infty \delta,
        \end{aligned}
    \end{equation}
    which proves the claim.
\end{proof}

We define the approximate Bellman operator $\hat{\mathcal{T}}$ as follows:
\begin{equation} \label{eq: approximate bellman equation}
\begin{aligned}
    & \hat{\mathcal{T}}Q(x) := \min_u ~\ell(x,u) + \sum_{m=1}^{M} p_m Q(Ax + Bu + \tilde{w}_m)  \\
    & ~~~~~~~~~~~~~~ \text{s.t.,} ~ x \in \mathcal{X}, ~ u \in \mathcal{U}, \\
    & ~~~~~~~~~~~~~~~~~~~ Ax + Bu \in \mathcal{SS}^\infty \ominus \mathcal{W}.
\end{aligned}
\end{equation}
\begin{proposition} \label{prop: Q = That Q}
    Let Assumptions \ref{assum: bounded noise}-\ref{assum: existence of Qstar} hold.
    $\forall x \in \mathcal{SS}^\infty$, $Q^\infty(x) \leq \hat{\mathcal{T}}Q^\infty(x)$.
\end{proposition}
\begin{proof}
    We sample within the safe set by \eqref{eq: sampling in SS} and add to the state data matrix as in \eqref{eq: state data}. Thus, as \( j \to \infty \), any $x \in \mathcal{SS}^\infty$ is a column vector of the state data matrix $\mathbf{X}^\infty$. 

    Consider $x_c$, a column vector of the state data matrix $\mathbf{X}^\infty$.
    For such $x_c$, as \( j \to \infty \), the input data matrix $\mathbf{U}^\infty$ contains an optimal input $\pi^\infty(x_c)$ by the construction of $\mathbf{U}^\infty$ from \eqref{eq: exp input matrix} and \eqref{eq: input data}. Therefore, for such $x_c$, $\hat{\mathcal{T}}Q^\infty(x_c)$ is an element of the cost data matrix $\mathbf{J}^\infty$ from \eqref{eq: cost data}. Considering that $Q^\infty(\cdot)$ is the optimal cost function of \eqref{eq: value function update} as $j$ tends towards $\infty$, it follows that $\hat{\mathcal{T}}Q^\infty(x_c) \geq Q^\infty(x_c)$. 
    
    Thus, $\forall x \in \mathcal{SS}^\infty$, $Q^\infty(x) \leq \hat{\mathcal{T}}Q^\infty(x)$.
\end{proof}

\begin{proposition} \label{prop: pic feasible}
    Let Assumptions \ref{assum: bounded noise}-\ref{assum: existence of Qstar} hold. Let $T=\max(T^\star, T^\infty)$. Then, there exists $\lambda \in (0,1)$ such that:
    \begin{equation} \label{eq: xkc in SSinfty}
        x_k^c = \lambda x^\infty_k + (1-\lambda)x^\star_k \in \mathcal{SS}^\infty, ~ \forall k \in \{0,\cdots,T\},
    \end{equation}
    which is a closed-loop state of the system \eqref{eq:system} controlled by the following control policy:
    \begin{equation} \label{eq: pic def}
        \pi^c(x_k^c) = \lambda \pi^\infty(x^\infty_k) + (1-\lambda) \pi^\star(x^\star_k).
    \end{equation}
    Moreover, $\pi^c(x_k^c)$ in \eqref{eq: pic def} is a feasible solution of \eqref{eq: approximate bellman equation} at $x=x_k^c$ for $k \in \{0,\cdots,T-1\}$.
\end{proposition}
\begin{proof}
First, since $x^\infty_k$ are in the relative interior of $\mathcal{SS}^\infty$ by Assumption \ref{assum: it converges to steady-state solution}, there exists $\lambda \in (0,1)$ that satisfies \eqref{eq: xkc in SSinfty}.

Second, by linearity of the system in \eqref{eq:system}, $x_k^c$ in \eqref{eq: xkc in SSinfty} is a closed-loop state of the system \eqref{eq:system} controlled by \eqref{eq: pic def}.

Third, we will show that $\pi^c(x_k^c)$ in \eqref{eq: pic def} is a feasible solution of \eqref{eq: approximate bellman equation} at $x=x_k^c$ for $k \in \{0,\cdots,T-1\}$. To do that, we need to verify the constraints of \eqref{eq: approximate bellman equation}.
For the state constraint, by \eqref{eq: xkc in SSinfty}, the closed-loop state $x_k^c \in \mathcal{SS}^\infty \subset \mathcal{X}$.
For the input constraint. $\pi^c(x) \in \mathcal{U}$. Note that $\pi^\infty(x) \in \mathcal{U}$, and $\pi^\star(x) \in \mathcal{U}$ for all $x \in \mathcal{SS}^\infty$. Moreover, $\mathcal{U}$ is a convex set. Thus, $\pi^c(x) \in \mathcal{U}$ for all $x \in \mathcal{SS}^\infty$.
For the terminal constraint, since the closed-loop states $x_k^c \in \mathcal{SS}^\infty$ for all $k \in \{0,\cdots,T\}$ from \eqref{eq: xkc in SSinfty}, $Ax_k^c + B\pi^c(x_k^c) + w_k \in \mathcal{SS}^\infty, ~\forall k \in \{0,\cdots,T-1\}, ~ \forall w_k \in \mathcal{W}$. 
Therefore, $\pi^c(\cdot)$ in \eqref{eq: pic def} is a feasible solution of \eqref{eq: approximate bellman equation} for $k \in \{0,\cdots,T-1\}$.
\end{proof}

\begin{theorem} \label{thm: nbd near optimal}
    Let Assumptions \ref{assum: bounded noise}-\ref{assum: existence of Qstar} hold. Let $\delta(w,M)$ satisfies the inequality in \eqref{eq: delta_bnd}, and $\delta = \max_{w \in \mathcal{W}} \delta(w,M)$. Then, the following holds:
    \begin{equation}
    \begin{aligned}
        & Q^\infty(x_S) \geq Q^\star(x_S) \geq Q^\infty(x_S) - c T L^\infty \delta,
    \end{aligned}
    \end{equation}
    where $c >1$ is a constant, and $T=\max(T^\star, T^\infty)$.
\end{theorem}
\begin{proof}
    First, we prove $Q^\star(x_S) \geq Q^\infty(x_S) - cT L^\infty \delta$. 
    Recall that $x_k^c = \lambda x^\infty_k + (1-\lambda)x^\star_k \in \mathcal{SS}^\infty$ with $\lambda \in (0, 1)$ as shown in \eqref{eq: xkc in SSinfty}. Evaluating the expected sum of the stage cost along the state trajectory $x_{0:T}^c$ yields:
    \begin{equation} \label{eq: Qc <= lam Qinf + 1-lam Qstar}
        \begin{aligned}
        & ~~~~ \mathbb{E}_{w_{0:T-1}}\big[{\scriptstyle\sum}_{k=0}^{T-1} \ell(x_k^c, ~ \pi^c(x_k^c))\big] \\
        & \leq \lambda \mathbb{E}_{w_{0:T-1}}\big[{\scriptstyle\sum}_{k=0}^{T-1} \ell(x_k^\infty, ~ \pi^\infty(x_k^\infty))\big] \\
        & ~~~~ + (1-\lambda) \mathbb{E}_{w_{0:T-1}}\big[{\scriptstyle\sum}_{k=0}^{T-1} \ell(x_k^\star, ~ \pi^\star(x_k^\star))\big] \\
        & \leq \lambda Q^\infty(x_S) + (1-\lambda) Q^\star(x_S),
        \end{aligned}
    \end{equation}
    where the first inequality is from the convexity of $\ell(\cdot, \cdot)$ from Assumption \ref{assum: Stage cost}, and the second inequality is from \eqref{eq: decreasing sequence sum of stage costs}.
    
    Furthermore, we have the following:
    \begin{equation} \label{eq: Qinf mid eq}
        \begin{aligned}
            & ~~~~ Q^\infty(x) \\
            & \leq \hat{\mathcal{T}}Q^\infty(x) \\
            & \leq \ell(x, \pi^c(x)) + {\scriptstyle\sum}_{m=1}^M p_m Q^\infty(Ax + B\pi^c(x) + \tilde{w}_m) \\
            & \leq \ell(x, \pi^c(x)) + \mathbb{E}_w [Q^\infty(Ax + B\pi^c(x) + w)] + L^\infty \delta,
        \end{aligned}
    \end{equation}
    where the first inequality is from Proposition \ref{prop: Q = That Q}, the second inequality is because $\pi^c(\cdot)$ is a feasible solution of \eqref{eq: approximate bellman equation} from Proposition \ref{prop: pic feasible} while $\hat{\mathcal{T}}Q^\infty(x)$ is an optimal function, and the third inequality is from Proposition \ref{prop: TQ tools}.
    By recursively applying \eqref{eq: Qinf mid eq}, we have the followings:
    \begin{equation} \label{eq: Qinf leq c sum + TLdelta}
        \begin{aligned}
            & ~~~~ Q^\infty(x_S)\\ 
            & \leq \ell(x_S, \pi^c(x_S)) + \mathbb{E}_{w_0} [Q^\infty(Ax_S + B\pi^c(x_S) + w_0)] + L^\infty \delta \\
            & \leq \ell(x_S, \pi^c(x_S)) + \mathbb{E}_{w_0} [\ell(x_1^c, \pi^c(x_1^c))] \\
            & ~~~ + \mathbb{E}_{w_{0:1}} [Q^\infty(Ax_1^c + B\pi^c(x_1^c) + w_1)] + 2L^\infty \delta \\
            & ~~~ \cdots\\
            & \leq \mathbb{E}_{w_{0:T-1}}\Big[{\scriptstyle\sum}_{k=0}^{T-1} \ell(x_k^c, ~ \pi^c(x_k^c)) \Big] + TL^\infty \delta,
        \end{aligned}
    \end{equation}
    where the last inequality holds because the task ends at time step $k=T$. Therefore, $\forall w_{T-1} \in \mathcal{W}$,  $Ax_{T-1}^c + B\pi^c(x_{T-1}^c) + w_{T-1} \in  \mathcal{O}$, and $Q^\infty(Ax_{T-1}^c + B\pi^c(x_{T-1}^c) + w_{T-1}) = 0$.

    From \eqref{eq: Qc <= lam Qinf + 1-lam Qstar} and \eqref{eq: Qinf leq c sum + TLdelta}, we have the following:
    \begin{equation*}
        \begin{aligned}
            & ~~~~~~~~~ Q^\infty(x) \leq \lambda Q^\infty(x_S) + (1-\lambda) Q^\star(x_S) + TL^\infty \delta \\
            & \iff Q^\star(x_S) \geq Q^\infty(x) - \frac{1}{1-\lambda} TL^\infty \delta,
        \end{aligned}
    \end{equation*}
    which proves the right-hand side of the inequality. 

    We now prove $Q^\infty(x_S) \geq Q^\star(x_S)$.
    As the MPC problem \eqref{eq:MPC discrete} imposes the state and input constraint and is always feasible from Proposition \ref{thm: recursive feasibility}, we have the following:
    \begin{equation}
        \begin{aligned}
            & x_{k+1}^\infty  = A x_k^\infty + B \pi^{\infty}(x_k^\infty) + w_k, \\
            & x_0^\infty = x_S, ~ w_k \sim p(w), \\
            & x_k^\infty \in \mathcal{SS}^\infty \subset \mathcal{X}, ~ \pi^{\infty}(x_k^\infty) \in \mathcal{U},  ~\forall w_k \in \mathcal{W},  \\
            & \forall k \geq 0.
        \end{aligned}
    \end{equation}
    This shows that the constraints of \eqref{eq:ftocp} are satisfied, thus $\pi^{\infty}(\cdot)$ is a feasible control policy of \eqref{eq:ftocp}.
    Moreover, $Q^\star(x_S)$ is an optimal value function of \eqref{eq:ftocp}.
    Thus, $Q^\infty(x_S) \geq Q^\star(x_S)$. 
\end{proof}
\begin{remark}
Theorem \ref{thm: nbd near optimal} establishes that the estimated value function at \( x_S \) converges to a neighborhood of the optimal value function of \eqref{eq: general format of problem} at \( x_S \), with a radius of \( cT L^\infty \delta \). The key factors that influence the radius are the task characteristics, including the task completion time \( T \) and the Lipschitz constant \( L^\infty \); the explored space, as \( c \) increases when the converged state sequence \( x^\infty_k \) approaches the boundary of \( SS^\infty \) (i.e., small \( \lambda \)); and the sampling density, since \( \delta \) decreases as the number of well-distributed samples increases.
\end{remark}
\begin{remark}
From Assumption \ref{assum: muw assumption} in Sec. \ref{sec: Approx}, the bound $\delta(w,M)$ in \eqref{eq: M1 M2 assumption} decreases as the number of discretized disturbances increases. Moreover, \( \delta = \max_{w \in \mathcal{W}} \delta(w,M) \). Thus, the steady-state function \( Q^\infty(\cdot) \) converges arbitrarily close to the optimal value function \( Q^\star(\cdot) \) with a sufficiently large number of sampled disturbances.
\end{remark}

\section{Simulation Results}
This section presents a numerical example based on the authors' other paper \cite{joa2024energy}, which solves the energy consumption minimization problem under localization uncertainty for connected, automated vehicles and has been investigated via experimental vehicle tests. Note that \cite{joa2024energy} builds on the theoretical framework presented in this paper. While the system and model for localization uncertainty remain unchanged, the parameters, stage cost, and target set have been modified for a simpler presentation. The parameters, stage cost, and target set used in the experimental test can be found in \cite{joa2024energy}.
\subsection{Problem Setup}
We consider an infinite horizon OCP \eqref{eq:ftocp} with the following parameters.
We consider the system:
\begin{equation} \label{eq: system param}
    \begin{aligned}
        & A=\begin{bmatrix} 1 & 1\\ 0 & 1 \end{bmatrix}, ~B=\begin{bmatrix} 0.5 \\ 1 \end{bmatrix}, \mathcal{W} = \Big\{w ~| ~||w||_\infty \leq \frac{1}{3}\Big\}, \\
        & x_S = \begin{bmatrix} 30 & 5 \end{bmatrix}^\top,
    \end{aligned}
\end{equation}
$p(w)$ is a uniform distribution at $\mathcal{W}$

The constraints are defined as follows:
\begin{equation*} 
    \begin{aligned}
        & \mathcal{X} = \Big\{x ~\Big|~ \begin{bmatrix} -10 \\ -40 \end{bmatrix} \leq x \leq \begin{bmatrix} 40 \\ 10 \end{bmatrix} \Big\}, ~\mathcal{U} = \{u ~|~ -5 \leq u \leq 5 \}.
    \end{aligned}
\end{equation*}

The target set $\mathcal{O}$ is set to the minimum robust positive invariant set for the closed-loop system \eqref{eq:system} controlled by linear policy $u(x) = K x$ where $K$ is the optimal LQR gain with the parameters $Q_\mathrm{LQR} = \begin{bmatrix} 1 & 0 \\ 0 & 1 \end{bmatrix}$ and $R_\mathrm{LQR} = 0.01$.
We utilize this minimum robust positive invariant set to design a tube MPC for initialization.

The stage cost is defined as follows:
\begin{equation} \label{eq: stage cost example}
    \begin{aligned}
        \ell(x,u) = \lVert x\rVert_\mathcal{O}^2 + \lVert u \rVert_\mathcal{KO}^2,
    \end{aligned}
\end{equation}
where
\begin{equation}
\begin{aligned}
    & \lVert x\rVert_\mathcal{O} = \min_{d\in \mathcal{O}} || x-d||_2, ~~\lVert u \rVert_\mathcal{KO} = \min_{d\in \mathcal{KO}} || u - d||_2.
\end{aligned}
\end{equation}
 Note that \eqref{eq: stage cost example} is differentiable as it is a sum of the square of the distance to the convex polytope.

We set the number of the discretized disturbances $M$ as $100$.
To calculate $p_m$ \eqref{eq: pm}, we utilize the method in Appendix \ref{sec: calculating muw LP}.

\subsection{Comparison between the proposed method and LMPC with a Certainty Equivalent cost}
This subsection compares the proposed method with an LMPC that minimizes a certainty equivalent cost, i.e., a nominal cost \cite{rosolia2018stochastic}. Both algorithms utilize data from previous tasks and update their safe set and terminal cost for each episode. However, the proposed method minimizes the expected sum of the stage cost, while the LMPC in \cite{rosolia2018stochastic} minimizes the sum of the certainty equivalent stage cost.

To evaluate the performance of each algorithm, we conduct 100 Monte-Carlo simulations of the system in \eqref{eq:system} with \eqref{eq: system param} as the parameters, controlled by each algorithm without updating the safe set and the value function for each episode $j$. We then investigate the sample mean of the total cost incurred by the system.
After completing the 100 Monte-Carlo simulations of episode $j$, we update the safe set and the value function using data from one of the 100 simulations. Then, we conduct another 100 Monte-Carlo simulations for episode $j+1$. We repeat this process iteratively to evaluate the algorithms.

Fig. \ref{fig:Performance comparison} shows the results of the comparison between the proposed method and \cite{rosolia2018stochastic}. 
At episode $10$, the proposed method achieves a total cost of $4712.27$, while \cite{rosolia2018stochastic} achieves $5359.98$, resulting in a $13.75 \%$ improvement for the proposed method.
\cite{rosolia2018stochastic} converges faster than the proposed method, as it has a longer horizon. However, its cost function differs from the expected cost, and its control policy parameterization makes the performance conservative. 
On the other hand, the proposed method takes a few more episodes to converge, but as it minimizes the approximation of the expected cost \eqref{eq: terminal cost approximation}, it shows better performance in terms of the realized cost.
\begin{figure}[ht]
\begin{center}
\includegraphics[width=0.7\linewidth,keepaspectratio]{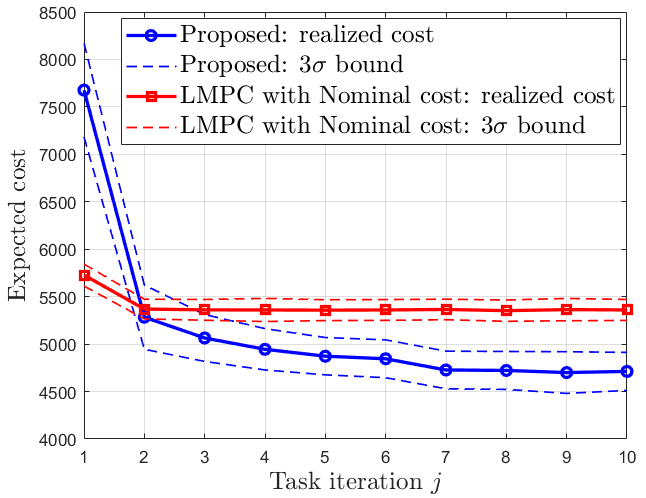}
\vspace{-1.0em}
\caption{sample mean of the realized cost for episode $j$}
\label{fig:Performance comparison}
\end{center}
\vspace{-2.0em}
\end{figure}

\subsection{Comparison between the proposed MPC and the Value iteration method}
In this subsection, we compare the value function of the proposed method algorithm with the value function calculated from the value iteration method in \cite{yang2020convex}, which grids the state space and approximates the Bellman operator as a convex optimization.
We grid the state space with a fixed grid size $1$.

\subsubsection{Value function comparison} In Fig. \ref{fig:value function comparison over j}, we present a comparison of the value function for both algorithms at $x_S = \begin{bmatrix} 30 & 5 \end{bmatrix}^\top$.
There are two notable points.
First, as proved in Proposition \ref{prop: nonincreasing value function}, the value function of the proposed algorithm at point $x_S$, $Q^j(x_S)$ monotonically decreases over episodes.
Second, as the episode $j$ increases, the error between $Q^j(x_S)$ and the value function at point $x_S$ calculated from the value iteration \cite{yang2020convex} decreases. 
\begin{figure}[ht]
\begin{center}
\includegraphics[width=0.7\linewidth,keepaspectratio]{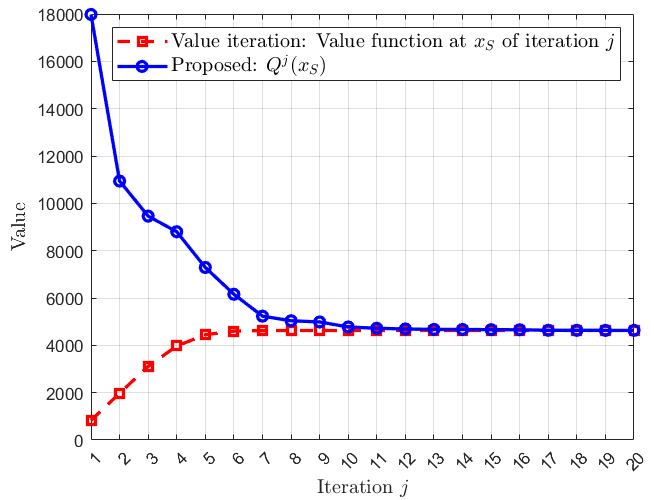}
\vspace{-1.0em}
\caption{Comparison: (a) Value function at $x_S$ calculated from the value iteration method and (b) Terminal cost $Q^j(x_S)$ of the proposed MPC.}
\label{fig:value function comparison over j}
\end{center}
\vspace{-1.0em}
\end{figure}

\subsubsection{Computation time comparison} In terms of computation time, the value iteration \cite{yang2020convex} took 11hr 42min for 20 iterations while the proposed method took 36min 53sec for 20 episodes including online and offline processes illustrated in Fig. \ref{fig:blockdiagram}. Thus, for this scenario, the proposed method is 19.04 times faster than the value iteration method. 

We analyze the time complexity of both algorithms. Solving a convex Quadratic Programming (QP) problem requires polynomial time \cite{ye1989extension}. Let \( f(n) \) represent the polynomial time complexity of solving a convex QP with input size \( n \).

As illustrated in Fig. \ref{fig:blockdiagram}, the proposed algorithm solves the MPC problem \eqref{eq:MPC discrete} (a convex QP) and then updates \(\mathcal{SS}^j\) and \( Q^j(\cdot) \) in the "Learning" block. The primary bottleneck is repeatedly solving the MPC problem, which must be solved \( |\mathbf{X}^{j}| - |\mathbf{X}^{j-1}| \) times. Specifically, for each newly added state in $|\mathbf{X}^{j}|$ \eqref{eq: state data}, the corresponding optimal solution is computed, as described in \eqref{eq: exp input matrix} and \eqref{eq: input data}. 
The input size of the MPC problem is \( n_u + M \times |\mathbf{X}^{j-1}| \), as it is reformulated into a single optimization problem by incorporating the value function \eqref{eq: value function update} into \eqref{eq:MPC discrete}. Therefore, the overall time complexity for each iteration is \( O(f(n_u + M |\mathbf{X}^{j-1}|) \times (|\mathbf{X}^{j}| - |\mathbf{X}^{j-1}|)) \). 

As the number of episodes \( j \) increases, the computation time grows polynomially because the size of the dataset \( |\mathbf{X}^{j-1}| \) increases according to \eqref{eq: state data}. Simultaneously, Proposition \ref{prop: nonincreasing value function} and the results shown in Fig. \ref{fig:Performance comparison} and Fig. \ref{fig:value function comparison over j} demonstrate an improvement in performance as \( j \) increases. Thus, there are trade-offs between computational efficiency and performance.

In \cite{yang2020convex}, a similar convex QP is solved for each gridded state (see \cite[eq. (3.1)]{yang2020convex}). However, unlike our algorithm, the value function in \cite{yang2020convex} is approximated over a fixed grid of states rather than from collected data. Let \( N_\text{grid} \) represent the number of grid points per dimension in the state space, resulting in a worst-case complexity of \( O(f(n_u + N_\text{grid}^{n_x}) \times N_\text{grid}^{n_x}) \). As the grid resolution increases, the computational cost grows exponentially, which is also shown in \cite[Table 1]{yang2020convex}.

A direct comparison of the two methods is not fair, as the baseline is a `global' method and the proposed one is a `local' method as illustrated in Fig. \ref{fig:value function comparison}. Nevertheless, we remark that the baseline method \cite{yang2020convex} exhibits exponential time complexity with respect to the state dimension \( n_x \). In contrast, the proposed method avoids this exponential scaling. While \( n_x \) does not explicitly appear in the big-O notation of our method, it affects the number of constraints, leading to a polynomial time complexity with respect to \( n_x \) in our approach.

\begin{figure}[ht]
\begin{center}
\includegraphics[width=0.48\linewidth,keepaspectratio]{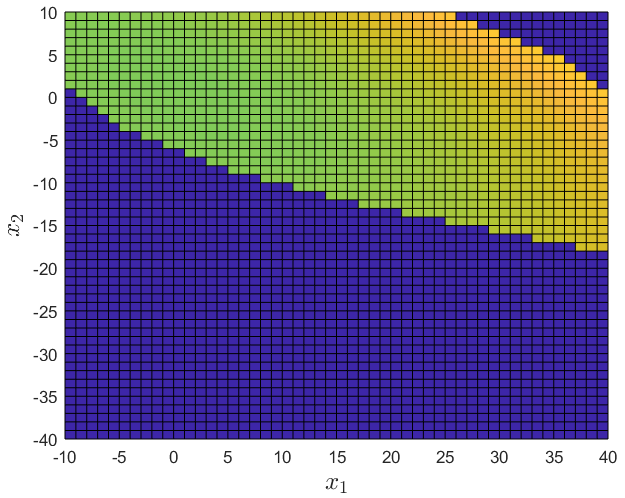}
\includegraphics[width=0.48\linewidth,keepaspectratio]{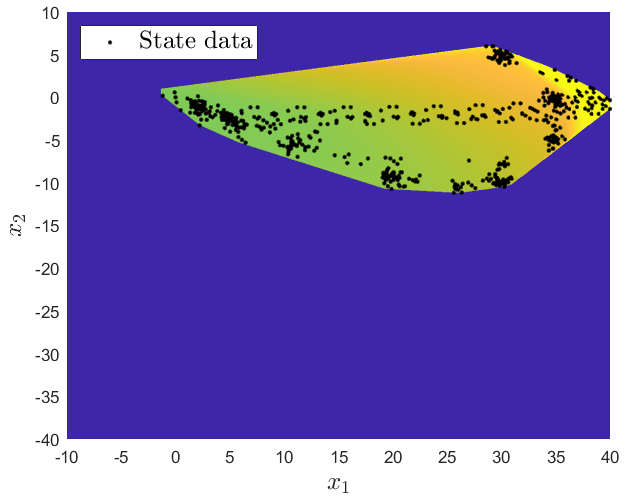}
\vspace{-1.0em}
\caption{Comparison: (a) Value function calculated from the value iteration method and (b) Terminal cost $Q^j(\cdot)$ of the proposed MPC. For both figures, the light region indicates the feasible region. The dark blue grid represents the infeasible region for the value iteration method. For the proposed method, the dark blue grid represents either the infeasible region or the region that has not been visited. The proposed method calculates the value function only within the local region based on available data, whereas value iteration computes the value function across the entire region.}
\vspace{-1.0em}
\label{fig:value function comparison}
\end{center}
\end{figure}

\section{Conclusion}
We propose a Model Predictive Control (MPC) method for approximately solving a class of infinite horizon optimal control problems aimed at minimizing the expected sum of stage costs for linear uncertain systems with bounded, state-additive random disturbances. 
The control objective is to steer the system from a given initial state to a target set.
The proposed method achieved robust satisfaction of state and input constraints and convergence in probability to the target set.
We further proved that as the number of episodes increases, the value function monotonically decreases, and under some assumptions, it converges to the optimal value function.
With numerical simulations, we demonstrated that the proposed method obtained a 13.75 \% performance improvement compared to the LMPC with a certainty equivalent cost. Moreover, compared to the value iteration, the proposed method is up to 19 times faster, while its value function converges towards that of the value iteration. 
Future work may explore data-driven approaches to solve the problem without relying on the assumption of an existing sub-optimal controller.

\printbibliography 
\balance
\appendix
\section{Appendix}
\subsection{Calculating the coefficient $\mu_m(w)$ in \eqref{eq: cvx sum with mu w} and $p_m$ \eqref{eq: pm}} \label{sec: calculating muw}
\subsubsection{Using a multi-parametric linear programming (mp-LP)}
We can calculate $\mu_m(w)$ by solving the following multi-parametric linear programming (mp-LP) as follows: 
\begin{equation} \label{eq: mplp for mu_w}
    \begin{aligned}
        & J^\star(w,M) = \min_{\mu_{1:M}(w)}~ {\scriptstyle\sum}_{m=1}^{M} \lVert \tilde{w}_m \rVert_2^2 \mu_m(w) \\
        & ~~~~~~~~~~~~~~~~~~ \textnormal{s.t.,} ~~~~  w = {\scriptstyle\sum}_{m=1}^M \mu_m(w) \tilde{w}_m, \\
        & ~~~~~~~~~~~~~~~~~~ \qquad ~~ {\scriptstyle\sum}_{m=1}^M \mu_m(w) = 1, \\
        & ~~~~~~~~~~~~~~~~~~ \qquad ~~ \mu_m(w) \geq 0, ~ \forall m \in \{1,\cdots,M\}. 
    \end{aligned}
\end{equation}
The solution of \eqref{eq: mplp for mu_w}, $\mu_m^\star(w)$, is given by an explicit piecewise affine function \cite{borrelli2017predictive}.

Several methods can compute the expectation \eqref{eq: pm} offline.
If probability distribution is available, one can use numerical integration. 
When it is not available, one can use recorded disturbances to perform the Monte Carlo integration.

\subsubsection{Using multiple linear programming (LP)} \label{sec: calculating muw LP}
As mp-LP might be computationally intractable for high-dimensional systems, we propose an alternate solution.
Considering that using Monte Carlo integration to calculate \eqref{eq: pm}, then we can write \eqref{eq: pm} for each sampled disturbance $d_n$ as follows:
\begin{equation} \label{eq: pm2}
    \begin{aligned}
        p_m = \frac{1}{N_\mathrm{sampled}} {\sum}_{n=1}^{N_\mathrm{sampled}} \mu_m(d_n)
    \end{aligned}
\end{equation}
for a large number $N_\mathrm{sampled}$.
For each $d_n$, we can calculate $\mu_m(d_n)$ by solving the following linear programming (LP): 
\begin{equation} \label{eq: lp for mu_w}
    \begin{aligned}
        & J^\star_\mathrm{LP}(d_n,M) = \min_{\mu_{1:M}(d_n)} ~ {\scriptstyle\sum}_{m=1}^{M} \lVert \tilde{w}_m \rVert_2^2 \mu_m(d_n) \\
        & ~~~~~~~~~~~~~~~~~~~~~~ \textnormal{s.t.,} ~~~~  d_n = {\scriptstyle\sum}_{m=1}^M \mu_m(d_n) \tilde{w}_m, \\
        & ~~~~~~~~~~~~~~~~~~~~~~ \qquad ~~ {\scriptstyle\sum}_{m=1}^M \mu_m(d_n) = 1, \\
        & ~~~~~~~~~~~~~~~~~~~~~~ \qquad ~~ \mu_m(d_n) \geq 0, ~ \forall m \in \{1,\cdots,M\}. 
    \end{aligned}
\end{equation}
After solving the LP \eqref{eq: lp for mu_w} for all $n \in \{1,\cdots,N_\mathrm{sampled}\}$, we calculate \eqref{eq: pm} as in \eqref{eq: pm2}.

\subsubsection{Prove that Assumption \ref{assum: muw assumption} is satisfied.}
We will prove that Assumption \ref{assum: muw assumption} is satisfied for the optimal solution of the problem \eqref{eq: mplp for mu_w}.
Based on this proof, it is straightforward that the optimal solution of \eqref{eq: lp for mu_w} satisfies Assumption \ref{assum: muw assumption}, which will be described at the end of this section.

It is trivial that there exists a bound $\delta(w,M)$ as the discretized disturbances $\tilde{w}_m, ~\forall m \in \{1,\cdots,M\}$ are located within the bounded set $\mathcal{W}$.
In the following, we will prove that as the number of the discretized disturbances goes to infinity, the bound $\delta(w,M)$ in \eqref{eq: delta_bnd} converges to zero.

We consider the problem \eqref{eq: mplp for mu_w}.
For a given $w$, let $\mu^\star_m(w)$ denote the optimal solution of the problem \eqref{eq: mplp for mu_w}.
Moreover, let $\tilde{p}_i \in \{\tilde{w}_1, \cdots, \tilde{w}_M\}$ for $i \in \{1, \cdots,M_p\}$ denote the discretized disturbances where the associated coefficient $\mu^\star_m(w)$ are non-zero.
Let $a_i(w), ~i \in \{1, \cdots,M_p\}$ denote associated coefficient of $\tilde{p}_i, ~i \in \{1, \cdots,M_p\}$.
Then, we have the following:
\begin{equation} \label{eq: J = aP}
    \begin{aligned}
        & J^\star(w,M) = {\scriptstyle\sum}_{i=1}^{M_p} \lVert \tilde{p}_i \rVert_2^2 a_i(w) \\
        & w = {\scriptstyle\sum}_{i=1}^{M_p} a_i(w) \tilde{p}_i.
    \end{aligned}
\end{equation}
\begin{proposition} \label{prop: no Q inside of P}
    For a given $w$, there do not exist discretized disturbances $\tilde{q}_j\in \{\tilde{w}_1, \cdots, \tilde{w}_M\}$ for $j \in \{1,\cdots,M_q\}$ that satisfies the following conditions:
    \begin{equation} \label{eq: assum muw prop}
        \begin{aligned}
            & w = {\scriptstyle\sum}_{j=1}^{M_q} b_j(w) \tilde{q}_j, \\
            & {\scriptstyle\sum}_{j=1}^{M_q} b_j(w) = 1, ~b_j(w) \geq 0, ~ \forall j \in \{1,\cdots,M_q\}, \\
            & \text{conv}\bigg\{\bigcup_{j=1}^{M_q} \tilde{q}_j\bigg\} \subset \text{relint}\bigg(\text{conv}\bigg\{\bigcup_{i=1}^{M_p} \tilde{p}_i\bigg\}\bigg).
        \end{aligned}
    \end{equation}
    where $\text{relint}(\cdot)$ represents the relative interior of the set.
\end{proposition}
\begin{proof}
    We use proof by contradiction. Suppose such discretized disturbances $\tilde{q}_j, ~j \in \{1,\cdots,M_q\}$ exist.
    Since $\tilde{q}_j$ is located within the relative interior of the convex hull of $\tilde{p}_i, ~i \in \{1, \cdots,M_p\}$ from \eqref{eq: assum muw prop}, each $\tilde{q}_j, ~j \in \{1, \cdots, M_q\}$ can be represented as a strict convex combination of $\tilde{p}_i, ~i \in \{1, \cdots,M_p\}$ as follows: 
    \begin{equation} \label{eq: cvx coeff p = q}
        \begin{aligned}
            & \tilde{q}_j = {\scriptstyle\sum}_{i=1}^{M_p} c_i \tilde{p}_i, ~j \in \{1, \cdots, M_q\}\\
            & {\scriptstyle\sum}_{i=1}^{M_p} c_i = 1, ~ c_i > 0.
        \end{aligned}
    \end{equation}
    Moreover, as a quadratic function is strictly convex, we have the following for all $j \in \{1, \cdots, M_q\}$:
    \begin{equation} \label{eq: q^2 < ap^2}
        \tilde{q}_j^\top \tilde{q}_j = \big({\scriptstyle\sum}_{i=1}^{M_p} c_i \tilde{p}_i \big)^\top \big({\scriptstyle\sum}_{i=1}^{M_p} c_i \tilde{p}_i \big) < {\scriptstyle\sum}_{i=1}^{M_p} c_i \tilde{p}_i^\top \tilde{p}_i.
    \end{equation}
    \eqref{eq: q^2 < ap^2} holds for any convex coefficient that satisfies \eqref{eq: cvx coeff p = q}.
    
    Now, consider the following two sets:
    \begin{equation}
        \begin{aligned}
            & \mathcal{P} = \text{conv}\bigg\{\bigcup_{i=1}^{M_p} \bigg\{ \begin{bmatrix}
                \tilde{p}_i^\top \tilde{p}_i \\ \tilde{p}_i
            \end{bmatrix} \bigg\}\bigg\}, ~ \mathcal{Q} = \text{conv}\bigg\{\bigcup_{j=1}^{M_q} \bigg\{ \begin{bmatrix}
                \tilde{q}_j^\top \tilde{q}_j \\ \tilde{q}_j
            \end{bmatrix} \bigg\}\bigg\},
        \end{aligned}
    \end{equation}
    Note that $J^\star(w,M) = {\scriptstyle\sum}_{i=1}^{M_p} \lVert \tilde{p}_i \rVert_2^2 a_i(w) = \sum_{i=1}^{M_p} \tilde{p}_i^\top \tilde{p}_i a_i(w)$ and $w = \sum_{i=1}^{M_p} a_i(w) \tilde{p}_i$ from \eqref{eq: J = aP}. This implies that a point $\begin{bmatrix} J^\star(w,M) & w^\top \end{bmatrix}^\top$ can be written as a convex combination of $\begin{bmatrix} \tilde{p}_i^\top \tilde{p}_i & \tilde{p}_i^\top \end{bmatrix}^\top, ~i \in \{1,\cdots,M_p\}$, which are the elements of $\mathcal{P}$.
    Thus, $\begin{bmatrix} J^\star(w,M) & w^\top \end{bmatrix}^\top \in \mathcal{P}$.
    
    Similarly, let $\hat{J}(w) = \sum_{j=1}^{M_q} \tilde{q}_j^\top \tilde{q}_j b_j(w)$. From \eqref{eq: assum muw prop}, we have that $w = \sum_{j=1}^{M_q} b_j(w) \tilde{q}_j$. Thus, $\begin{bmatrix} \hat{J}(w) & w^\top \end{bmatrix}^\top \in \mathcal{Q}$.
    Furthermore, from \eqref{eq: q^2 < ap^2}, all vectors $\begin{bmatrix} \tilde{q}_j^\top \tilde{q}_j & \tilde{q}_j^\top \end{bmatrix}^\top, ~j \in \{1,\cdots,M_q\}$ are located below the set $\mathcal{P}$ with respect to the first coordinate. Based on these facts, $\begin{bmatrix} \hat{J}(w) & w^\top \end{bmatrix}$ is located below the set $\mathcal{P}$ with respect to the first coordinate. Thus $\hat{J}(w) < J^\star(w,M)$.

    Finally, $\hat{J}(w) < J^\star(w,M)$ contradicts that $J^\star(w,M)$ is the optimal cost of the problem \eqref{eq: mplp for mu_w}.
    This proves the claim.
\end{proof}

Proposition \ref{prop: no Q inside of P} implies the following:
Consider the scenario when we continuously sample the disturbance within $\mathcal{W}$ while maintaining the previous $M_1$ samples until there exist discretized disturbances $\tilde{q}_j, ~j \in \{1,\cdots,M_q\}$ that satisfy \eqref{eq: assum muw prop} for some $w$.
Then, we can find a smaller bound than the bound $\delta(w,M_1)$.
Specifically, we have the following:
\begin{equation}
    \exists N > M_1, ~\forall M_2 \geq N, ~\delta(w,M_1) > \delta(w,M_2).
\end{equation}
Therefore, the optimal solution of \eqref{eq: mplp for mu_w} satisfies Assumption \ref{assum: muw assumption}. 
When using \eqref{eq: lp for mu_w}, Proposition \ref{prop: no Q inside of P} can be proved by replacing $w$ with $d_n$ and $J^\star(w,M)$ with $J^\star_\text{LP}(d_n,M)$.

\subsection{Initialization using Tube MPC} \label{sec: tube MPC initialize}
For initialization, we have the following suboptimal controller that satisfies Assumption \ref{assum: initial solution}.
\begin{equation} \label{eq:tube ocp}
\begin{split}
    & \bar{V}^0_{0 \rightarrow T} ({x}_S) = \min_{\bar{u}_{0:T}^0} ~~ 0 \\
    & \qquad \qquad \qquad  \textnormal{s.t.,} ~~ \bar{x}_{k+1}^0 = A \bar{x}_{k}^0 + B \bar{u}_{k}^0,  \\
    & \qquad \qquad \qquad \qquad \, \bar{x}_{0}^0 = {x}_S,\\
    & \qquad \qquad \qquad \qquad \, \bar{x}_{k+1}^0 \in \mathcal{X} \ominus \mathcal{E}, ~ \Bar{u}_{k}^0 \in \mathcal{U} \ominus K\mathcal{E},\\
    & \qquad \qquad \qquad \qquad \, \bar{x}_{T+1}^0 \in \mathcal{O} \ominus \mathcal{E},  ~k \in \{0,\cdots,T\},
\end{split}
\end{equation}
where the superscript $^0$ is to denote an initialization, $\bar{x}_{k}^0$ and $\bar{u}_{k}^0$ are a nominal state and a nominal input at time step $k$, respectively, and  $\mathcal{E}$ is a robust positive invariant set for the autonomous system $x_{k+1} = (A+BK)x_k + w_k, ~ w_k \in \mathcal{W}$ and $K$ is a given state feedback gain such that $A+BK$ is Hurwitz. 
\begin{remark}
This problem is a tube MPC \cite{mayne2005robusto} with horizon length $N=T+1$. By \cite[Prop.1]{mayne2005robusto}, the state $x_k^0$ of the closed-loop system \eqref{eq:system} under the control policy $\pi_\mathrm{Tube}(x_k^0) = \bar{u}_{k}^{0\star} + K(x_k^0 - \bar{x}_{k}^{0\star})$ always satisfies the state constraint $\mathcal{X}$ while $\pi_\mathrm{Tube}(x_k^0) \in \mathcal{U}$.
Moreover, the terminal state $x_{T+1}^0 \in \mathcal{O}$.
\end{remark}
The initial safe set is defined in \eqref{eq: initial terminal constraint app}. 
The initial value function is defined in \eqref{eq: initial value function}.



\subsubsection{Construct the initial cost-to-go data vector $\mathbf{J}^{0}$}
In this subsection, we describe an initial cost-to-go data vector denoted as $\mathbf{J}^{0}$, where $i$-th element of $\mathbf{J}^{0}$, i.e., $[\mathbf{J}^{0}]_i$, corresponds to $[\mathbf{X}^{0}]_i$ and $[\mathbf{U}^{0}]_i$.
For $k \in \{0,\cdots,T+1\}$, we compute the cost-to-go value $J^0_k$ for each tube $\mathcal{E}_k$ \eqref{eq: tubes}, i.e., $f: \mathcal{E}_k \rightarrow J^0_k$, by solving the following optimization problem:
\begin{equation} \label{eq: worst case ctg for initialization}
\begin{split}
    &  J_k^0 = \max_{x} ~\ell(x, \pi_\mathrm{Tube}(x)) + J_{k+1}^0 \\
    & ~~~~~~~~ \textnormal{s.t.,} ~~x \in \mathcal{E}_k, 
\end{split}
\end{equation}
Note that $J_{T+1}^0 = 0$ since $\mathcal{E}_{T+1}\subseteq \mathcal{O}$.
If $[\mathbf{X}^{0}]_i$ falls within a tube $\mathcal{E}_k$, we assign the corresponding cost-to-go value $J_k^0$ to it.

As the objective function in \eqref{eq: worst case ctg for initialization} is convex and $\mathcal{E}_k$ is a polytope, a vertex of $\mathcal{E}_k$ is the global maximizer of \eqref{eq: worst case ctg for initialization} \cite[Prop. 1.3.4]{bertsekas2009convex}. 
Thus, we can solve this maximization problem \eqref{eq: worst case ctg for initialization} by enumerating the cost-to-go values of all vertices of $\mathcal{E}_k$.
The enumeration number is $T \cdot l_w$.

For the first $l_\mathcal{O}$ elements of the cost data vector $\mathbf{J}^{0}$, their corresponding state data in $\mathbf{X}^{0}$ are the vertices of $\mathcal{O}$.
Thus, the first $l_\mathcal{O}$ elements of the $\mathbf{J}^{0}$ are zeros.
The remaining elements of the $\mathbf{J}^{0}$ are assigned based on the tube to which the respective state belongs, i.e., if $[\mathbf{X}^{0}]_i \in \mathcal{E}_k$, the corresponding $[\mathbf{J}^{0}]_i = J_k^0$.
To sum up, we define the initial cost data vector $\mathbf{J}^{0}$ as follows:
\begin{equation} \label{eq: initial cost data}
\begin{split}
    & \mathbf{J}^{0} = [\underbrace{0, 0,..., 0}_{l_\mathcal{O}}, ~ \underbrace{J_0^0,\cdots, J_0^0}_{l}, \cdots, \underbrace{J_{T+1}^0,\cdots, J_{T+1}^0}_{l}].
\end{split}
\end{equation}
    \subsubsection{Prove that the cost-data vector \eqref{eq: initial cost data} satisfies the condition \eqref{eq: Q0 condition}}
    For the simplicity of the proof, we define the function $J^0(\cdot)$ as follows:
    \begin{equation} \label{eq: def of J^0}
    J^0(x)= 
    \begin{cases}
        J_k^0, & \text{if } x \in  \mathcal{E}_k, ~k\in \{0,\cdots,T\},\\
        0, & \text{else if } x \in \mathcal{O},\\
        \infty,              & \text{otherwise}.
    \end{cases}
    \end{equation}
    Thus, for any $i$-th column of $\mathbf{X}^0$, the following is satisfied:
    \begin{equation}
        J^0([\mathbf{X}^0]_i) = [\mathbf{J}^0]_i.
    \end{equation}

    By construction of the tubes $\mathcal{E}_k, ~ k\in \{0,\cdots,T\}$ \eqref{eq: tubes} and by Assumption \ref{assum: Robust Positive Invariance}, for any state in $\mathbf{X}^0$, the corresponding input in $\mathbf{U}^{0}$ will steer the system \eqref{eq:system} inside one of the tubes \eqref{eq: tubes} or the target set $\mathcal{O}$. Therefore, we have that:
    \begin{equation} \label{eq: const of J next state}
        J^0(A[\mathbf{X}^{0}]_i + B[\mathbf{U}^{0}]_i + w) = const. < \infty, ~ \forall w \in \mathcal{W}.
    \end{equation}
    
    In \eqref{eq: worst case ctg for initialization}, we define the cost-to-go value by taking a maximization over the tube while setting the cost-to-go value in the target set $\mathcal{O}$ to zero.
    This implies the following inequality is satisfied for all elements of $\mathbf{X}^{0}$:
    \begin{equation} \label{eq: J inequality}
    \begin{aligned}
        & J^0([\mathbf{X}^{0}]_i) \\
        & \geq \ell([\mathbf{X}^{0}]_i,[\mathbf{U}^{0}]_i) + J^0(A[\mathbf{X}^{0}]_i + B[\mathbf{U}^{0}]_i + w), ~ \forall w \in \mathcal{W},
    \end{aligned}
    \end{equation}
    From \eqref{eq: const of J next state}, we know that $J^0(A[\mathbf{X}^{0}]_i + B[\mathbf{U}^{0}]_i + w)$ is constant for any realization of $w \in \mathcal{W}$.
    Thus, we can rewrite \eqref{eq: J inequality} as:
    \begin{equation} \label{eq: J inequality reform}
    \begin{aligned}
        & J^0([\mathbf{X}^{0}]_i) \\
        & \geq \ell([\mathbf{X}^{0}]_i,[\mathbf{U}^{0}]_i) + {\scriptstyle\sum}_{m=1}^M p_m J^0(A[\mathbf{X}^{0}]_i + B[\mathbf{U}^{0}]_i + \tilde{w}_m),
    \end{aligned}
    \end{equation}
    Remind that $\tilde{w}_m$ is the $m$-th discretized disturbance.
    For brevity of the proof, we will denote $A[\mathbf{X}^{0}]_i + B[\mathbf{U}^{0}]_i + \tilde{w}_m$ as $x_{i,m}^+$ in the rest of the proof.

    
    Next, we want to prove:
    \begin{equation} \label{eq: J^0 > Q^0 claim}
        J^0(x_{i,m}^+) \geq Q^0(x_{i,m}^+).
    \end{equation}
    As aforementioned in \eqref{eq: const of J next state}, $x_{i,m}^+$ belongs to either the target set $\mathcal{O}$ or one of the tubes $\mathcal{E}_k(\bar{x}_k^{0\star}), ~ k\in{1,\cdots,T}$. Let us denote the set to which $x_{i,m}^+$ belongs as $S_{i,m}^+$.
    Then, we can express $x_{i,m}^+$ as a convex combination of the vertices of the set $S_{i,m}^+$ which are the elements of $\mathbf{X}^0$. 
    Let $\bm{\lambda}^v$ is the associated convex coefficients, i.e., $\mathbf{X}^0 \bm{\lambda}^v = x_{i,m}^+$.
    Moreover, $J^0(x_{i,m}^+)$ is constant over the tube $S_{i,m}^+$.
    Thus, we have the following:
    \begin{equation}
    \begin{aligned}
        & \mathbf{J}^0 \bm{\lambda}^v = J^0(x_{i,m}^+), \\
        & \mathbf{X}^0 \bm{\lambda}^v = x_{i,m}^+, \\
        & \bm{\lambda}^v \geq \bm{0}, ~ \bm{1}^\top\bm{\lambda}^v = 1.  
    \end{aligned}
    \end{equation}
    This implies $\bm{\lambda}^v$ is a feasible solution of the mp-LP problem in \eqref{eq: initial value function} and thus $J^0(x_{i,m}^+)\geq Q^0(x_{i,m}^+)$.
    Therefore, from \eqref{eq: J inequality reform} and $J^0(x_{i,m}^+)\geq Q^0(x_{i,m}^+)$ \eqref{eq: J^0 > Q^0 claim}, the claim is proved.

\subsection{Robust Reachable Set \cite{borrelli2017predictive}} \label{sec: robust reachable set}
\begin{definition}
(Robust Successor Set): Given a policy $\pi(\cdot)$ and the closed-loop system $x_{k+1} = Ax_k + B \pi(x_k) + w_k, ~w_k \in \mathcal{W}$, the robust successor set from the $\mathcal{S}$ is defined as:
\begin{equation*}
    \begin{split}
        & \mathrm{Succ}(S, \mathcal{W}, \pi) = \{x_{k+1}~|~\exists x_k \in \mathcal{S}, \exists w_k \in \mathcal{W}, \qquad \\
        & \quad \,\,\, \qquad \qquad \qquad ~~~~~~~~~~ x_{k+1} = A x_k + B \pi(x_k) + w_k \}.
    \end{split}
\end{equation*}
\end{definition}
Given the set $\mathcal{S}$, the robust successor set $\mathrm{Succ}(S, \mathcal{E}_n, \pi)$ denotes the set of the states that the uncertain autonomous system can reach in one time step. 
\begin{definition} \label{def:Robust Reachable Set}
($N$ step Robust Reachable Set): Given a policy $\pi(\cdot)$ and the closed-loop system $x_{k+1} = Ax_k + B \pi(x_k) + w_k, ~w_k \in \mathcal{W}$, we recursively define the $N$ step robust reachable set from the $\mathcal{S}$ subject to a constraint $x_k \in \mathcal{X}$ as:
\begin{equation*}
    \begin{split}
        & \mathcal{R}_{0}(\mathcal{S}, \mathcal{W}, \pi) = \mathcal{S}, \\
        & \mathcal{R}_{i+1}(\mathcal{S}, \mathcal{W}, \pi) = \mathrm{Succ}(\mathcal{R}_{i}(\mathcal{S}, \mathcal{W}, \pi),\mathcal{W}, \pi) \medcap \mathcal{X}, \\
        & i=0,1,..,N-1. 
    \end{split}
\end{equation*}
\end{definition}

\subsection{Proof of Proposition \ref{prop: SS0 is rpi}} \label{sec: proof of rpi}
We will prove the claim by induction.
Consider the case $j=0$.
Since the terminal constraint in \eqref{eq:tube ocp} is $\bar{x}_{T+1}^0 \in \mathcal{O} \ominus \mathcal{E}$, the realized state ${x}_{T+1}^0$ of the closed-loop system \eqref{eq:system} controlled by $\pi_\mathrm{Tube}(\cdot)$ is in the target set $\mathcal{O}$.
Moreover, for all $N \in \{0,\cdots,T+1\}$, $N$-th tubes $\mathcal{E}_N$ \eqref{eq: tubes} contains a $N$-step robust reachable set $\mathcal{R}_N(x_S, \mathcal{W}, \pi_\mathrm{Tube})$. 
Thus, by \cite[Proposition 1]{rosolia2021robust}, $\mathcal{SS}^{0}$ is a robust positive invariant set for the closed-loop system \eqref{eq:system} controlled by the policy $\pi^{0}_\mathrm{Q}(\cdot)$ \eqref{eq: safe policy}.

Assuming for a particular $j-1 \geq 0$, $\mathcal{SS}^{j-1}$ is a robust positive invariant set for the system \eqref{eq:system} controlled by the policy $\pi^{j-1}_\mathrm{Q}(\cdot)$.
We want to show that $\mathcal{SS}^{j}$ is a robust positive invariant set for the system \eqref{eq:system} controlled by the policy $\pi^{j}_\mathrm{Q}(\cdot)$.
From the terminal constraint of the \eqref{eq:MPC discrete} and the control policy \eqref{eq:MPC policy}, we have that,
\begin{equation} \label{eq: prop support eq 1}
    Ax_k^j + B\pi^{j\star}(x_k^j) + w \in \mathcal{SS}^{j-1}, ~\forall w \in \mathcal{W}.
\end{equation}
Similarly, for any element $x_\mathrm{exp}^j$ of $\mathbf{X}_{k,\mathrm{exp}}^j$ and $u_\mathrm{exp}^j = \pi^{j\star}(x_\mathrm{exp}^j)$ of $\mathbf{U}_{k,\mathrm{exp}}^j$ where $k \in \{0,\cdots,T^j\}$, we have that,
\begin{equation} \label{eq: prop support eq 2}
    Ax_\mathrm{exp}^j + Bu_\mathrm{exp}^j + w \in \mathcal{SS}^{j-1}, ~\forall w \in \mathcal{W}.
\end{equation}
Moreover, by the induction hypothesis and \eqref{eq: safe policy}, $\forall x \in \mathcal{SS}^{j-1}$,
\begin{equation} \label{eq: prop support eq 3}
    A\underbrace{\mathbf{X}^{j-1} \bm{\lambda}^{j-1\star}}_{x} + B\underbrace{\mathbf{U}^{j-1} \bm{\lambda}^{j-1\star}}_{\pi^{j-1}_\mathrm{Q}(x)} + w \in \mathcal{SS}^{j-1}, ~\forall w \in \mathcal{W}.
\end{equation}
The above equations \eqref{eq: prop support eq 1}, \eqref{eq: prop support eq 2} and \eqref{eq: prop support eq 3} imply that $\forall x=\mathbf{X}^{j} \bm{\lambda}^{j\star}\in\mathcal{SS}^{j}$ at episode $j$, satisfy the following equation:
\begin{equation}
    A\underbrace{\mathbf{X}^{j} \bm{\lambda}^{j\star}}_{x} + B\underbrace{\mathbf{U}^{j} \bm{\lambda}^{j\star}}_{\pi^{j}_\mathrm{Q}(x)} + w \in \mathcal{SS}^{j-1}, ~\forall w \in \mathcal{W}.
\end{equation}
Also, by \eqref{eq: terminal constraint update}, $\mathcal{SS}^{j-1} \subset \mathcal{SS}^j$.
Therefore, $\forall x\in\mathcal{SS}^{j}$, $Ax + B\pi^{j}_\mathrm{Q}(x) + w\in\mathcal{SS}^{j-1} \subset \mathcal{SS}^{j}, ~\forall w \in \mathcal{W}$. 


\begin{IEEEbiography}[{\includegraphics[width=1in,height=1.25in,clip,keepaspectratio]{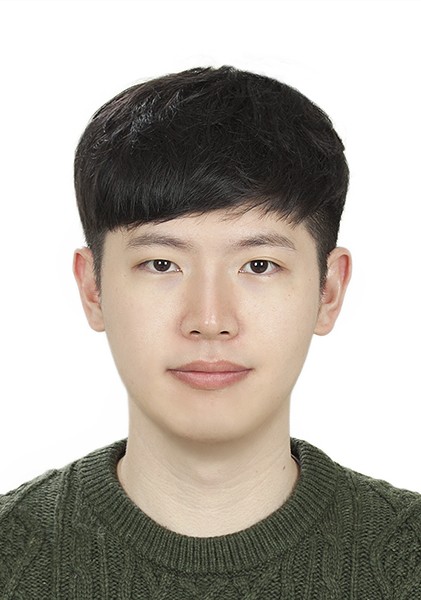}}]{Eunhyek Joa} (Member, IEEE)  received the B.S. and M.S degrees in mechanical engineering from Seoul National University, Seoul, Korea, in 2014 and 2016, respectively. He is working toward a Ph.D. in mechanical engineering at the University of California at Berkeley, sponsored by the Korean Government Scholarship Program.
His research interests include data-driven control, constrained optimal control, and model predictive control and their applications to advanced automotive control.
\end{IEEEbiography}

\begin{IEEEbiography}[{\includegraphics[width=1in,height=1.25in,clip,keepaspectratio]{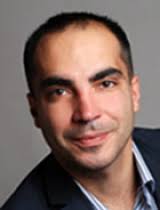}}]{Francesco Borrelli} (Fellow, IEEE) received the Laurea degree in computer science engineering from the University of Naples Federico II, Naples, Italy, in 1998, and the Ph.D. degree from ETH Zürich, Zürich, Switzerland, in 2002.
He is a Professor with the Department of Mechanical Engineering, University of California at Berkeley. He is the author of more than 100 publications in the field of predictive control and the book Constrained Optimal Control of Linear and Hybrid Systems (Springer Verlag). His research interests include constrained optimal control, model predictive control, and its application to automotive control and energy-efficient building operation. 
\end{IEEEbiography}

\end{document}